\documentclass[12pt, reqno]{amsart} 
\usepackage{amsthm,amssymb}
\usepackage{graphicx}
\usepackage{enumerate}
\usepackage{color}
\usepackage[margin=3cm, a4paper]{geometry}

\def\jb#1{\langle#1\rangle} 
\def\norm#1{\|#1\|}
\def\brk#1{\left(#1\right)}
\def\normb#1{\big\|#1\big\|} 

\def\normo#1{\left\|#1\right\|}
\def\aabs#1{\left|#1\right|}
\def\abs#1{\big|#1\big|}
\newcommand{\les}{\lesssim} 
\newcommand{\ges}{\gtrsim}
\newcommand{\wt}{\widetilde}
\newcommand{\Lr}{{\mathcal{L}}}
\newcommand{\BR}[1]{\left[#1\right]}
\newcommand{\Br}[1]{\left\{#1\right\}}

\newcommand{\F}{\mathcal{F}}

\newcommand{\C}{\mathbb{C}}
\newcommand{\N}{\mathbb{N}} 
\newcommand{\R}{\mathbb{R}}
\newcommand{\Z}{\mathbb{Z}}
\newcommand{\E}{\mathbb{E}} 
\newcommand{\cir}{\mathbb{S}}

\newcommand{\al}{\alpha} 
\newcommand{\be}{\beta}
\newcommand{\ga}{\gamma} 

\newcommand{\e}{\varepsilon} 

\newcommand{\om}{\omega} 

\newcommand{\te}{\theta}

\newcommand{\x}{\xi}

\newcommand{\ro}{\rho} 
\newcommand{\ft}{{\mathcal{F}}}
\newcommand{\Hl}{{\mathcal{H}}}

\newcommand{\De}{\Delta} 
\newcommand{\Des}{\Delta_\sigma}

\newcommand{\p}{\partial} \newcommand{\na}{\nabla}
\newcommand{\re}{\mathop{\mathrm{Re}}}
\newcommand{\im}{\mathop{\mathrm{Im}}}

\newcommand{\lec}{\lesssim} 
 
   \newcommand{\I}{\infty}

\newcommand{\EQ}[1]{\begin{equation}\begin{split}
      #1 \end{split}\end{equation}} 
\newcommand{\EQN}[1]{\begin{equation*}\begin{split}
      #1 \end{split}\end{equation*}} 
 \newcommand{\Del}[1]{}

\newcommand{\pt}{&} \newcommand{\pr}{\\ &} \newcommand{\pq}{\quad}
 \newcommand{\prq}{\\ &\quad} \newcommand{\prQ}{\\
  &\qquad} 

\numberwithin{equation}{section}

\newtheorem{thm}{Theorem}[section] 
\newtheorem{cor}[thm]{Corollary}
\newtheorem{lem}[thm]{Lemma}

\theoremstyle{remark} 
\newtheorem{rem}[thm]{Remark}
\newtheorem*{exam*}{Examples}

\begin{document}

\title[3D Gross-Pitaevskii equation]{Scattering for the
3D Gross-Pitaevskii equation} 
  
\author[Z. Guo, Z. Hani, K. Nakanishi]{Zihua Guo, Zaher Hani, Kenji Nakanishi} 
\address{School of Mathematical Sciences, Monash University, VIC 3800, Australia}
\email{zihua.guo@monash.edu}

\address{School of Mathematics, Georgia Institute of Technology, Atlanta, GA 30332, USA}
\email{hani@math.gatech.edu}

\address{Department of Pure and Applied Mathematics Graduate School of Information Science and Technology Osaka University, Suita, Osaka 565-0871, JAPAN}
\email{nakanishi@ist.osaka-u.ac.jp}

\thanks{Z. H. was supported by a Sloan Fellowship, National Science Foundation grant DMS-1600561, and a startup fund from Georgia Institute of Technology.}

\begin{abstract}
We study the Cauchy problem for the 3D Gross-Pitaevskii equation.  The global well-posedness in the natural energy space was proved by G\'erard \cite{Gerard}. In this paper we prove scattering for small data in the same space with some additional angular regularity, and in particular in the radial case we obtain small energy scattering. 
\end{abstract}

\maketitle

\tableofcontents

\section{Introduction}

\subsection{Main results}
In this paper we study the asymptotical behaviour of the solution to the 3D Gross-Pitaevskii (GP) equation
\begin{align}\label{eq:3dGP}
i\psi_t+\Delta \psi=(|\psi|^2-1)\psi, \quad \psi:\R^{1+3}\to \C
\end{align}
with the boundary condition 
\begin{align}\label{eq:bdcon}
\lim_{|x|\to \infty }\psi=1.
\end{align}
It is easy to see the equation \eqref{eq:3dGP} is formally equivalent to cubic Schr\"odinger equations.  Indeed, let $\phi=e^{-it}\psi$, then $\phi$ solves 
 \begin{align}\label{eq:3dcNLS}
i\phi_t+\Delta \phi=|\phi|^2\phi.
\end{align}
The NLS \eqref{eq:3dcNLS} has been extensively studied for initial data in $H^s$ or other spaces with boundary condition $\lim_{|x|\to \infty }\phi=0$.  However, the nonzero boundary condition \eqref{eq:bdcon} or more generally  $\lim_{|x|\to \infty} |\psi|=1$, also arises naturally in physical contexts such as Bose-Einstein condensates, superfluids and nonlinear optics, or in the hydrodynamic interpretation of NLS (see \cite{FS}).  In these physical settings, $\phi=e^{-it}$ (or $\psi=1$) corresponds to a stationary, constant-density condensate. On the other hand, these type of boundary conditions bring remarkable effects on the space-time behaviour of the solutions.

The GP equation \eqref{eq:3dGP} has rich structures.  
Let $u=\psi-1$ be the perturbation from the equilibrium. Then $u$ satisfies
\EQ{\label{eq:GP2}
i\partial_tu+\Delta u-2\re u=&u^2+2|u|^2+|u|^2u,\\
 u|_{t=0}=&u_0.
}
We have conservation of the energy: if $u$ is a smooth solution to \eqref{eq:GP2} then
\begin{align*}
E(u):=&\int_{\R^3}|\nabla u|^2+\frac{(|u|^2+2\re u)^2}{2}dx=\int_{\R^3}|\nabla \psi|^2+\frac{(|\psi|^2-1)^2}{2}dx\\
=&E(u_0).
\end{align*}
An unconditional global well-posedness for \eqref{eq:GP2} in the energy space $\E$ was proved by G\'erard \cite{Gerard}, where
\begin{align}
\E:=\{f\in \dot H^1(\R^3): 2\re f+|f|^2\in L^2(\R^3)\}
\end{align}
with the distance $d_\E(f,g)$ defined by
\begin{align}
d_\E(f,g)^2=\norm{\nabla (f-g)}_{L^2}^2+\frac{1}{2}\normb{|f|^2+2\re f-|g|^2-2\re g}_{L^2}^2.
\end{align}
Note that $(\E,d_\E)$ is a complete metric space.
Global well-posedness for \eqref{eq:GP2} in a smaller space $H^1(\R^3)$ was previously proved in \cite{BSaut} where a-priori $L^2$-bound was derived by the Gronwall inequality:
\begin{align}\label{eq:L2bd}
\norm{u(t)}_{L_x^2}\leq C\norm{u(0)}_{L_x^2}\cdot e^{Ct}.
\end{align} 

The global behavior of the solutions to GP equation was also extensively studied. It was known (see \cite{BSaut, Chiron}) that there is a family of traveling wave solutions of the form $\psi(t,x)=v_c(x-ct)$ with finite energy for $0<|c|<\sqrt{2}$, where $\sqrt{2}$ is the sound speed. Moreover, it was proved in \cite{BGS} that there is a lower bound on the energy of all possible travelling waves for \eqref{eq:3dGP} in three dimensions:
\begin{align}
E^*:=\inf \{E(\psi-1)| \psi(t,x)=v(x-ct)\mbox{ solves }\eqref{eq:3dGP}\mbox{ for some } c\}>0,
\end{align}
and it was conjectured in \cite{BGS} that $E^*$ is the threshold for the global dispersive solutions, namely the solutions to GP with energy below the threshold $E^*$ should be dispersive (they scatter to a linear solution). The scattering problems for suitable small solutions were studied by the third named author, Gustafson and Tsai  (see \cite{GNT1,GNT2,GNT3}).
In these works they proved that under some decay and regularity conditions (weighted Sobolev space) on the initial data, the small solutions scatter to the solution of the linearized equation in dimension three and higher. In 3D, due to the quadratic nonlinearity, the weighted space rather than the energy space was essentially needed in these works.

The purpose of this paper is to study the scattering problem for the 3D GP equation without assuming decay conditions, and thus improve the previous results in \cite{GNT3}. Instead of decay conditions we assume additional angular regularity, namely we have additional information for $D_\sigma u_0$, the angular derivative of initial data $u_0$ (See Section 1.2 for the precise definitions). 
In particular, we can obtain scattering for small radial data in the energy space $\E$. This will open the possibility to study the large data problem. We note that in the radial case there is no non-trivial travelling wave solution, 
so it is reasonable to conjecture that in the radial case the
smallness condition is not needed. To state our results, we linearise the equation \eqref{eq:GP2} around 0 as in \cite{GNT1}.  By the diagonalising transform
\begin{align}
u=u_1+iu_2\longrightarrow v=v_1+iv_2:=u_1+iUu_2,
\end{align}
we get the equation for $v$:
\begin{align}\label{eq:GP3}
i\partial_t v-Hv=U(3u_1^2+u_2^2+|u|^2u_1)+i(2u_1u_2+|u|^2u_2),
\end{align}
where
\begin{align}\label{eq:UH}
U:=\sqrt{-\Delta(2-\Delta)^{-1}}, \quad H:=\sqrt{-\Delta (2-\Delta)}.
\end{align}
We may rewrite the nonlinear terms in \eqref{eq:GP3} as functions of $v$ by the simple inverse $u_1+iu_2=v_1+iU^{-1}v_2$.  Then we get a nonlinear Schr\"odinger type equations with quadratic and cubic terms. Note that $H$ behaves like Schr\"odinger at high frequency and wave at low frequency. We will study systematically the estimates for the linear propagator $e^{-itH}$ in Section 2. 
The main result of this paper is

\begin{thm}\label{thm:main}
There exists $\delta>0$ such that for any $u_0\in \E^1=\{f\in \E, D_\sigma f\in \E\}$ with $E(u_0)+E(D_\sigma u_0)\leq \delta$ there exists a unique global solution $u\in C(\R:\E^1)$ to \eqref{eq:GP2}. Moreover, let $T(u)=u_1+{(2-\Delta)^{-1}u_2^2}+iUu_2$, then there exists $\phi_{\pm}\in H^{1,1}=\{f\in H^1, D_\sigma f\in H^1\}$ such that
\begin{align}
\lim_{t\to \pm \infty}(\norm{T(u)-e^{-itH}\phi_\pm}_{H^1}+\norm{D_\sigma(T(u)-e^{-itH}\phi_\pm)}_{H^1})=0.
\end{align}
\end{thm}

In particular, if $u_0\in \E$ is radial, then $D_\sigma u_0=0$ and hence $u_0\in \E^1$. Thus by the above theorem we obtain

\begin{thm}\label{thm:rad}
There exists $\delta>0$ such that for any $u_0\in \E$, radial,  with $E(u_0)\leq \delta$, there exists a unique global solution $u\in C(\R:\E)$ to \eqref{eq:GP2}. Moreover, there exists $\phi_{\pm}\in H^1$ such that
\begin{align}
\lim_{t\to \pm \infty}\normo{u_1+(2-\Delta)^{-1}u_2^2+iUu_2-e^{-itH}\phi_\pm}_{H^1}=0.
\end{align}
\end{thm}

\begin{rem}
We can obtain the decay for quadratic terms of $u_1$ (however not true for $u_2$):
\begin{align*}
\lim_{t\to \pm \infty} \norm{(2-\Delta)^{-1}u_1^2}_{H^1}=0,
\end{align*}
from which we can transfer the asymptotic behaviour of $T(u)$ to $u$ in the energy space $\E$ by inverting $T$. Indeed, we can prove 
\begin{align*}
\lim_{t\to \pm \infty}&(\norm{u_1-\re(e^{-itH}\phi_\pm)}_{\dot H^1}+\norm{u_2-U^{-1}\im(e^{-itH}\phi_\pm)}_{\dot H^1}\\
&\qquad+\norm{2u_1+|u|^2-2\re(e^{-itH}\phi_\pm)}_{L_x^2})=0.
\end{align*}
and hence $\lim_{t\to \pm \infty}d_\E(u,T^{-1}(e^{-itH}\phi_\pm))=0$.
\end{rem}

Now we discuss briefly the strategy of the proof of the theorem. 
There are three main ingredients.  

(1) The first one is to handle the singularity at zero frequency $\xi=0$ caused by $U^{-1}$. The equation \eqref{eq:GP3} does not have derivative as $U$ is $0$-order multiplier. However, $U^{-1}\approx D^{-1}$ at low frequency. To eliminate the singularity, we rely on some normal form type transform
\begin{align*}
z=u+B(u,u)
\end{align*}
where $B$ is a bilinear form. The normal form transform is of the same type, and is usually used to deal with the loss of derivative by exploiting the non-resonance of the nonlinearity and transfer the low-order terms to high-order terms, for example see \cite{GN,GNW}. In the previous works \cite{GNT1,GNT2,GNT3}, the following transform
\begin{align}\label{eq:ntran}
z=z_1+iz_2=u_1+\frac{u_1^2+u_2^2}{2-\Delta}+iUu_2
\end{align}
was used. Under this transform, the GP equation reduces to a better system without singularity at low frequency.  In this paper, we will derive a different type of transform to achieve some subtle cancellations, although our transformation results in a more complicated system with quartic and quintic terms. See Section 2. 

(2) The second is to handle the quadratic terms in 3D.  We know that 3D Quadratic nonlinear Schr\"odinger equation is mass-subcritical and the usual Strichartz analysis does not give scattering in the Sobolev space.  That is the main reason why the previous works require weighted Sobolev space.  In this paper we will handle the quadratic interactions in the Sobolev space with additional angular regularity. These ideas are motivated by the works for 3D Zakharov system \cite{GN, GLNW, Guo}. This approach allows us to prove scattering in the energy space in the radial case. 

(3) Another essential difficulty is the lack of $L^2$ control in the energy space. By the energy conservation we only have control on the $L^2$-norm of $\re u$ while not on $\im u$, although we can control $\norm{u}_2$ for any finite time by \eqref{eq:L2bd}.  To prove scattering for small radial data in the smaller space $H^1(\R^3)$, our proof can be simplified a lot and we can use the transform \eqref{eq:ntran}.  However, the energy space $\E$ is natural and more important as it opens the possibility to study the large data problem.  Working in the energy space, we have weak control on the low frequency component  of $\im u$. Fortunately enough, we can overcome this difficulty by choosing proper nonlinear transform which results in some subtle cancellations similar to null structure.

\subsection{Notations and definitions}

For $a,b\in \R$, $a\les b$ means there exists $C>0$ such that $a\leq Cb$, $a\ges b$ means $b\les a$, and $a\sim b$ means $a\les b\les a$. We denote $\jb{a}=(2+a^2)^{1/2}$. When $r\in [1,\infty]$, we denote by $r'$ the conjugate number, namely $\frac{1}{r'}+\frac{1}{r}=1$.

Let $\eta: \R\to [0, 1]$ be an even, non-negative,
smooth and radially decreasing function which is supported in $\{\xi:|\xi|\leq 8/5\}$ and
$\eta\equiv 1$ for $|\xi|\leq 5/4$.
For $k\in \Z$ we define
$\chi_k(\xi)=\eta(\xi/2^k)-\eta(\xi/2^{k-1})$ and $\chi_{\leq
k}(\xi)=\eta(\xi/2^k)$. Then we define the Littlewood-Paley projectors $P_k, P_{\leq k}$ on $L^2(\R^d)$ by
\[\widehat{P_ku}(\xi)=\chi_k(|\xi|)\widehat{u}(\xi),\quad \widehat{P_{\leq
k}u}(\xi)=\chi_{\leq k}(|\xi|)\widehat{u}(\xi),\] where we use $\ft(f)$ and $\widehat{f}$ to denote
the Fourier transform of $f$.

Let $\cir^{d-1}\subset \R^d$ be the unit sphere, endowed with the standard metric
$g$ and measure $d\sigma$. Let $\Des$ be the Laplace-Beltrami operator
on $\cir^{d-1}$. Define $\Lambda_\sigma=\sqrt{1-\Delta_\sigma}$ and $D_\sigma=\sqrt{-\Delta_\sigma}$. For $1\leq i,j\leq d$,
$X_{ij}=x_i\partial_j-x_j\partial_i$. It is well-known that for
$f\in C^2(\R^d)$
\[\Des(f)(x)=\sum_{1\leq i<j\leq d}X_{ij}^2(f)(x).\]
Denote $L_\sigma^p=L_\sigma^p(\cir^{d-1})=L^p(\cir^{d-1}:d\sigma)$ and 
$\Hl_p^s=\Hl_p^s(\cir^{d-1})=\Lambda_\sigma ^{-s}L_\sigma^p$.

$L^p(\R^d)$ denotes the usual Lebesgue space,
and $\Lr^p(\R^+)=L^p(\R^+:\ro^{d-1}d\ro)$.
$\Lr_\ro^pL_\sigma^q $ and $\Lr_\ro^p\Hl^s_q$ are
Banach spaces defined  by the  following norms
\[\norm{f}_{\Lr_{\ro}^pL_\sigma^q}=\big\|{\norm{f(\ro\sigma)}_{L_\sigma^q}}\big\|_{\Lr_{\ro}^p},\quad
\norm{f}_{\Lr_{\ro}^p\Hl^s_q}=\big\|{\norm{f(\ro\sigma)}_{\Hl^s_q}}\big\|_{\Lr_{\ro}^p}.\]
$H^s_p$, $\dot{H}_p^s$ ($B^s_{p,q}$,
$\dot{B}^s_{p,q}$) are the usual Sobolev (Besov) spaces on $\R^d$. We simply write $B^s_{p}=B^s_{p,2}$ and
$\dot{B}^s_{p}=\dot{B}^s_{p,2}$.
Let $X$ be a Banach space on $\R^d$, we define $XL_\sigma^q$ by the norm $\norm{f}_{XL_\sigma^q}=\normb{\norm{f(|x|\sigma)}_{L_\sigma^q}}_{X}$. Note that $L_x^pL_\sigma^q=\Lr_\ro^p L_\sigma^q$.  We also define $L_t^qX$
to be the space-time space on $\R\times \R^d$ with the norm
$\norm{u}_{L_t^qX}=\big\|\norm{u(t,\cdot)}_X\big\|_{L_t^q}$.

\section{Normal-form type transform}

In this section we derive a normal-form type transform. 
The GP equation \eqref{eq:GP2} for $u=u_1+iu_2$ can be rewritten as follows
\EQ{\label{eq:GP2.1}
 \dot u_1 &= -\De u_2 + 2(u_1+|u|^2/2)u_2,
 \\ -\dot u_2 &= (2-\De)u_1 + 3u_1^2 + u_2^2 + |u|^2u_1
 \\&=(2-\De)(u_1)+2u_1^2+u_2^2+(2u_1+|u|^2)^2/4-|u|^4/4.}
Note that $2u_1+|u|^2\in L^2_x$ is bounded by the conserved energy. 
In view of the equation of $u_2$, we first make the following change of variables 
\EQ{
 z_1 := u_1 + \frac{2u_1^2+u_2^2}{2-\De}, \quad z_2=u_2.}
Then the equations \eqref{eq:GP2.1} are simplified to 
\EQ{
  \dot u_1 &= -\De z_2+2z_1z_2+2Rz_2,
 \\ -\dot z_2 &= (2-\De)z_1 + z_1^2+2z_1R+R^2 - |u|^4/4,}
where 
\EQ{\label{eq:R}
 R\pt:=u_1+\frac{|u|^2}{2}-z_1=
 \frac{|u|^2}{2}-\frac{2u_1^2+u_2^2}{2-\De}
 =\frac{-\De u_2^2}{2(2-\De)}-\frac{(2+\De)u_1^2}{2(2-\De)}.}
By direct computations we have
\EQ{
 \pt u_1\dot u_1=-z_1\De u_2 + \frac{2u_1^2+u_2^2}{2-\De}\De u_2+2u_1(z_1+R)u_2,
 \pr u_2\dot u_2=-(2-\De)(u_2z_1)-z_1\De u_2-2\na z_1\cdot\na u_2 - u_2(z_1+R)^2+u_2|u|^4/4,}
We continue to compute
\EQ{
 \dot z_1 &=\dot u_1+2\frac{2u_1\dot u_1+u_2\dot u_2}{2-\De}
 \pr = -\De z_2 +\frac{2}{2-\De}[-3z_1\De z_2-2\na z_1\cdot\na z_2]
 \prq +2\BR{Rz_2+\frac{2}{2-\De}\Br{\frac{2u_1^2+u_2^2}{2-\De}\De z_2}}+\frac{8}{2-\De}(u_1z_1z_2)
 \prq +\frac{2}{2-\De}[4u_1Rz_2- u_2(z_1+R)^2+z_2|u|^4/4].}
Keeping it in mind that the low frequency component of $u_2$ are problematic as we have weak control, we 
see that the quadratic terms are good since it involves $\nabla u_2$.  The crucial novelty is that the first two cubic terms have cancellation in $u_2^3$ at low frequency: 
\EQ{\label{eq:N31}
N^1_3(u):\pt= 2\BR{Ru_2+\frac{2}{2-\De}\Br{\frac{2u_1^2+u_2^2}{2-\De}\De u_2}}
 \pr=\Br{\frac{-\De u_2^2}{2-\De}-\frac{2+\De}{2-\De}u_1^2}u_2+\frac{4}{2-\De}\Br{\frac{2u_1^2+u_2^2}{2-\De}\De u_2}
 \pr=\frac{-\De u_2^2}{2-\De}u_2+\frac{4}{2-\De}\Br{\frac{u_2^2}{2-\De}\De u_2}
 \prq+\frac{4}{2-\De}\BR{\frac{2u_1^2}{2-\De}\De u_2}-\BR{\frac{2+\De}{2-\De}u_1^2}u_2
 \pr=(2-\Delta)^{-1}\Br{-2u_2|\na u_2|^2+\frac{3\De u_2^2}{2-\De}\De u_2+\frac{2\na\De u_2^2}{2-\De}\na u_2}
 \prq+\frac{4}{2-\De}\BR{\frac{2u_1^2}{2-\De}\De u_2}-\BR{\frac{2+\De}{2-\De}u_1^2}u_2.
}
Therefore, the system reduces to
\EQ{
   \dot z_1 &=-\De z_2 +\frac{2}{2-\De}[-3z_1\De z_2-2\na z_1\cdot\na z_2]+N_3^1(u)
 \prq +\frac{2}{2-\De}[4u_1z_1z_2- z_1^2u_2+4u_1Rz_2-2u_2z_1R-u_2R^2+z_2|u|^4/4],
 \\ -\dot z_2 &= (2-\De)z_1 + z_1^2+2z_1R+R^2 - |u|^4/4.}
By the diagonalizing transform $m=m_1+im_2:=z_1+iUz_2$, namely
\begin{align}\label{eq:ntran2}
m_1+im_2=u_1 + \frac{2u_1^2+u_2^2}{2-\De}+iUu_2
\end{align}
we get
\EQ{\label{eq:GPm}
i\partial_t m-Hm&=N_2(m)+N_3(m,u)+N_4(m,u)+N_5(m,u),\\
u_1&=m_1-\frac{2u_1^2+u_2^2}{2-\De},\\
u_2&=U^{-1}m_2,
}
where
\EQ{
N_2(m,u)&=U(m_1^2)+\frac{2i}{2-\De}[-3m_1\De u_2-2\na m_1\cdot\na u_2],\\
N_3(m,u)&=U(2m_1R)+iN_3^1(u)+\frac{2i}{2-\De}[4u_1m_1u_2- m_1^2u_2],\\
N_4(m,u)&=U(R^2 - |u|^4/4)+\frac{2i}{2-\De}[4u_1Ru_2- 2u_2m_1R],\\
N_5(m,u)&=\frac{2i}{2-\De}[-u_2R^2+u_2|u|^4/4],
}
and here $R$ is given by \eqref{eq:R} and $N_3^1$ is given by \eqref{eq:N31}, $U$ and $H$ are defined by \eqref{eq:UH}. We will study the equations \eqref{eq:GPm}. 

\begin{rem}
We would like to compare the nonlinear transform \eqref{eq:ntran2} and \eqref{eq:ntran}. Under the transform \eqref{eq:ntran}, the following simplified equations are derived
\begin{align*}
iz_t-Hz=&-2i U(u_1^2)-4\jb{\nabla}^{-2}\nabla\cdot (u_1\nabla u_2)+[-iU(|u|^2u_1)+U^2(|u|^2u_2)].
\end{align*}
The quadratic terms are fine, however, the cubic terms $u_2^2 \Delta u_2$ or $u_2^2 u_1$ when $u_2$ has very low frequency are problematic. The null structure in \eqref{eq:N31} are essential to our analysis although the resulted systems \eqref{eq:GPm} are more complicated. 
\end{rem}

\begin{rem}
The main difficulty in the scattering problems is the zero frequency. 
In that respect, the worst terms in \eqref{eq:GPm} for the energy scattering on $\R^3$ are the following, respectively in $N_3^1(u)$ and in $N_5(u)$: 
\EQ{
 N_3^c(u):=-u_2\frac{2+\De}{2-\De}u_1^2, 
 \pq N_5^c(u):=-\frac{i}{2(2-\De)}[u_2|u|^4].}
More precisely, the parts of $O(u_2^5)$ are the most singular in low frequency among all the nonlinear terms. In $N_3^c(u)$, the quintic term is coming from the transform 
\EQ{
 u_1 = m_1 - \frac{2u_1^2+u_2^2}{2-\De}.}
Ignoring all the other terms after the substitution yields a system 
\EQ{
 \ddot u_2= (2-\De)\BR{\De u_2 + u_2\frac{2+\De}{2-\De}\Br{\frac{u_2^2}{2-\De}}^2} - \frac12u_2^5.}
Then the zero frequency limit $\x\to 0$, replacing $2-\De$ by $2$, leads to 
\EQ{
 \ddot u_2 - 2\De u_2 = 0.}
namely the free wave equation on $\R^3$. In other words, the equation \eqref{eq:GPm} for $m$ has another hidden cancellation in the level of quintic interactions. 
\end{rem}

\section{Generalized Strichartz estimates}

In this section we derive the spherically averaged Strichartz estimates for the propagator $e^{-itH}$ which are crucial in our proof of Theorem \ref{thm:main}. Our proof uses the methods in \cite{GLNW, Guo}. We will prove the following lemma.

\begin{lem}\label{lem:sphstri} 
Assume $k\in \Z$, $\phi\in L^2(\R^3)$ and $r\in (10/3,\infty]$. 
Then
\begin{align}
\norm{e^{-itH}P_{k}\phi}_{L_t^2L_x^rL_\sigma^2}\les& C_k(2,r)
\norm{\phi}_{L_x^2(\R^3)},
\end{align}
where 
\begin{align}
C_k(2,r)=
\begin{cases}
2^{k(\frac{1}{2}-\frac{3}{r})}, \qquad k\geq 0;\\
2^{k(2-\frac{7}{r})}, \qquad  k<0, \frac{10}{3}<r<4;\\
2^{k(1-\frac{3}{r})}, \qquad k<0,r>4;\\
\jb{k}2^{\frac{k}{4}}, \quad k<0, r=4.
\end{cases}
\end{align}
\end{lem}

\begin{cor}\label{cor:sphstrqr}
Assume $k\in \Z$, $\phi\in L^2(\R^3)$ and $2\leq q , r\leq \infty$ 
Then
\begin{align}
\norm{e^{-itH}P_{k}\phi}_{L_t^qL_x^rL_\sigma^2}\les& C_k(q,r)
\norm{\phi}_{L_x^2(\R^3)},
\end{align}
where 
\begin{align}\label{eq:Ckqr}
C_k(q,r)=
\begin{cases}
2^{k(\frac{3}{2}-\frac{3}{r}-\frac{2}{q})}, \qquad k\geq 0, \frac{2}{5}<q(\frac{1}{2}-\frac{1}{r})\leq 1;\\
2^{k(\frac{3}{2}-\frac{3}{r}-\frac{1}{q})}, \qquad k<0,\frac{1}{2}<q(\frac{1}{2}-\frac{1}{r})\leq 1;\\
2^{k(\frac{7}{2}-\frac{7}{r}-\frac{3}{q})}, \qquad  k< 0, \frac{2}{5}<q(\frac{1}{2}-\frac{1}{r})<\frac{1}{2};\\
\jb{k}^{2/q}2^{\frac{k}{4}}, \qquad k<0, q(\frac{1}{2}-\frac{1}{r})=\frac{1}{2}.
\end{cases}
\end{align}
\end{cor}
\begin{proof}
This corollary follows immediately from interpolation between Lemma \ref{lem:sphstri} and the trivial estimate
$\norm{e^{-itH}P_{k}\phi}_{L_t^\infty L_x^2L_\sigma^2}\les
\norm{\phi}_{L_x^2(\R^3)}$.
\end{proof}

Instead of just proving Lemma \ref{lem:sphstri}, we will derive the generalized Strichartz estimates for a class of dispersive equations as we think they may be useful in the other occasions. Consider a class of dispersive
equations:
\begin{equation}\label{eq:wD}
i\partial_tu=-\om(D)u+f,\quad u(0)=u_0(x),
\end{equation}
where $D=\sqrt{-\Delta}$, $\om: \mathbb{R}^+ \to \R$ is $C^3$ smooth,
$f(x,t),\ u(x,t):\R^d\times \R \to \C,\ d\geq 2$, and
$\om(D)u={\F}^{-1}\om(|\xi|)\F u$.
Equation \eqref{eq:wD} contains many dispersive equations and in particular the one  $\om(D)=H$ considered in this paper. In \cite{GPW}, a systematic study of the dispersive estimates for the propagator $e^{it\om(D)}$ was carried out under some conditions on the asymptotic behavior of $\om$ around $0$ and $\infty$.  
We will assume similar conditions in this paper.  For $k\in \Z$, let $I_k=(2^{k-1},2^{k+1})$ and we say the following conditions:

\begin{enumerate}[{H}1(k):]
\item There exists $\al\in \R$ such that
\begin{align}
|\om'(r)|\ges 2^{k(\al-1)} \quad \mbox{ for   } r\in I_k.
\end{align}

\item H1(k) holds and there exists  $\be$, with $\be\leq \al$ for $k\geq 0$ and $\be\geq \al$ for $k<0$, such that
\begin{align}
|\om''(r)|\ges 2^{k(\be-2)} \quad \mbox{   for   } r\in I_k.
\end{align}
Moreover, $\frac{|\om''(r)|}{|\om'(r)|}\les 2^{-k}$ for $r\in I_k$, and $\omega'''$ changes its sign finite times in $I_k$.

\item The following inequality holds
\begin{align}
\om'(r)\om''(r)>0, \quad r\in I_k.
\end{align}
\end{enumerate}
All the implicit constants in the above conditions are independent of $k$. Our results depend only on these constants but not the specific form of $\omega$. These conditions are easy to verify, showing the dispersive effect of $\om$ at the frequency of scale $\sim 2^k$. We note that if $\om$ satisfies H2(k), then
\begin{align}\label{eq:sgnma}
k(\al-\be)\geq 0.
\end{align}
Also note that if in H1(k) and H2(k) the lower bounds are replaced by $|\om'(r)|\sim r^{\al-1}$ and $|\om''(r)|\sim r^{\be-2}$, then we must have $\al\geq \be$ for $k\geq 0$ and $\al\leq \be$ for $k<0$, and hence \eqref{eq:sgnma} holds.

The operator $H$ corresponds to $\om(r)=r\sqrt{2+r^2}$.   
Simple computations show that 
\begin{align}\label{eq:Halbe}
\begin{split}
\om'(r)=&\frac{2+2r^2}{\sqrt{2+r^2}}\sim \jb{r},\\
\om''(r)=&\frac{6r+2r^3}{(2+r^2)^{3/2}}\sim r \jb{r}^{-1},\\
\om'''(r)=&\frac{12}{(2+r^2)^{5/2}}\sim \jb{r}^{-5}.
\end{split}
\end{align}
Thus, we see $\om$ satisfies H2(k) and H3(k) with $\al=\be=2$ for $k\geq 0$, $\al=1, \be=3$ for $k<0$.
We list more examples for the potential applications in other occasions.

\begin{exam*}
\quad

\begin{enumerate}

\item Schr\"odinger type: $\om=r^a$, $a>0$.

\noindent If $a>1$, $\om$ satisfies H2(k) and H3(k) with $\al=\be=a$ for $k\in \Z$. If $a=1$, $\om$ satisfies H1(k) with $\al=\be=1$ for $k\in \Z$. If $0<a<1$, $\om$ satisfies H2(k) with $\al=\be=a$ for $k\in \Z$.

\item Klein-Gordon: $\om=(1+r^2)^{1/2}$.
\begin{align*}
\om'(r)=&r(1+r^2)^{-1/2},\\
\om''(r)=&(1+r^2)^{-3/2}.
\end{align*}
Then $\om$ satisfies H2(k) and H3(k) with $\al=1, \be=-1$ for $k\geq 0$, $\al=\be=2$ for $k<0$.

\item Beam equation: $\om=(1+r^4)^{1/2}$.
\begin{align*}
\om'(r)=&2r^3(1+r^4)^{-1/2},\\
\om''(r)=&(6r^2+2r^6)(1+r^4)^{-3/2}.
\end{align*}
Then $\om$ satisfies H2(k) and H3(k) with $\al=\be=2$ for $k\geq 0$, $\al=\be=4$ for $k<0$.

\item Fourth-order Schr\"odinger: $\om=r^2+\e r^4$.
\begin{align*}
\om'(r)=&2r+4\e r^3,\\
\om''(r)=&2+12\e r^2.
\end{align*}
Then $\om$ satisfies H2(k) and H3(k) with $\al=\be=2$ for $k\in \Z$ uniformly with respect to $\e\geq 0$.
\end{enumerate}
\end{exam*}

The main result of this section is the following theorem which implies Lemma \ref{lem:sphstri} by taking $\alpha=\beta=2$ for $k\geq 0$, and $\alpha=1,\beta=3$ for $k<0$.

\begin{thm}\label{thm} 
(1) Assume $k\in \Z$, $\om$ satisfy $\mathrm{H1(k)}$, $d\geq 3$ and $\frac{2d-2}{d-2}<r\leq\infty$. Then we have: for any $\phi \in L_x^2(\R^d)$
\begin{align}
\norm{e^{it\om(D)}P_k\phi}_{L_t^2L_x^rL_\sigma^2}\les& 2^{k(\frac{d}{2}-\frac{d}{r})}2^{-k\al/2}
\norm{\phi}_{L_x^2(\R^d)},
\end{align}
and for $d=2$, {$2\leq q\leq r\leq \infty$} and $1/q<1/2-1/r$ , we have: for any $\phi \in L_x^2(\R^2)$
\begin{align}
\norm{e^{it\om(D)}P_k\phi}_{L_t^qL_x^rL_\sigma^2}\les& 2^{k(\frac{d}{2}-\frac{d}{r})}2^{-k\al/q}
\norm{\phi}_{L_x^2(\R^2)}.
\end{align}

(2) Assume $k\in \Z$, $\om$ satisfy $\mathrm{H2(k)}$, $d\geq 2$ and $\frac{4d-2}{2d-3}<r\leq \frac{2d-2}{d-2}$. Then we have: for any $\phi\in L^2(\R^d)$, radial 
\begin{align}\label{eq:thm2(2)}
\norm{e^{it\om(D)}P_k\phi}_{L_t^2L_x^rL_\sigma^2}\les& \norm{\phi}_{L_x^2}\times
\begin{cases}
2^{k\theta_k(2,r)}, \qquad\qquad \frac{4d-2}{2d-3}<r< \frac{2d-2}{d-2};\\
\jb{k(\al-\beta)}2^{k\theta_k(2,r)}, \quad r=\frac{2d-2}{d-2};
\end{cases}
\end{align}
where
\begin{align}
\theta_k(2,r)=\frac{d}{2}-\frac{d}{r}-\frac{\beta}{2}-(\al-\beta)\brk{\frac{d-1}{2}-\frac{d-1}{r}}.
\end{align}
Moreover, if $\om$ also satisfies an additional condition $\mathrm{H3(k)}$, then \eqref{eq:thm2(2)} holds in the non-radial case, namely for all $\phi \in L^2(\R^d)$.
\end{thm}

By interpolation between the trivial estimate $\norm{e^{it\om(D)}P_{k}\phi}_{L_t^\infty L_x^2L_\sigma^2}\les
\norm{\phi}_{L_x^2(\R^3)}$ and Theorem \ref{thm} 
we get

\begin{cor}\label{Ltq} 
(1) Assume $k\in \Z$, $\om$ satisfy $\mathrm{H1(k)}$, $d\geq 2$, $2\leq q, r \leq \infty$ and $q(\frac{1}{2}-\frac{1}{r})>\frac{1}{d-1}$. Then we have: for any $\phi \in L_x^2(\R^d)$
\begin{align}
\norm{e^{it\om(D)}P_k\phi}_{L_t^qL_x^rL_\sigma^2}\les& 2^{k(\frac{d}{2}-\frac{d}{r})}2^{-k\al/q}
\norm{\phi}_{L_x^2(\R^d)},
\end{align}

(2) Assume $k\in \Z$, $\om$ satisfy $\mathrm{H2(k)}$, $d\geq 2$, $2\leq q, r \leq \infty$ and $\frac{2}{2d-1}<q(\frac{1}{2}-\frac{1}{r})\leq \frac{1}{d-1}$. Then we have: for any $\phi\in L^2(\R^d)$, radial 
\begin{align}\label{eq:cor2(2)}
\norm{e^{it\om(D)}P_k\phi}_{L_t^qL_x^rL_\sigma^2}\les& \norm{\phi}_{L_x^2}\times
\begin{cases}
2^{k\theta_k(q,r)}, \qquad\qquad \frac{2}{2d-1}<q(\frac{1}{2}-\frac{1}{r})<\frac{1}{d-1};\\
\jb{k(\alpha-\beta)}^{\frac{2}{q}}2^{k\theta_k(q,r)}, \quad q(\frac{1}{2}-\frac{1}{r})=\frac{1}{d-1};
\end{cases}
\end{align}
where
\begin{align}
\theta_k(q,r)=\frac{d}{2}-\frac{d}{r}-\frac{\beta}{q}-(\al-\beta)\brk{\frac{d-1}{2}-\frac{d-1}{r}}.
\end{align}
Moreover, if $\om$ also satisfies an additional condition $\mathrm{H3(k)}$, then \eqref{eq:cor2(2)} holds in the non-radial case, namely for all $\phi \in L^2(\R^d)$.
\end{cor}

\begin{rem}
To prove \eqref{eq:thm2(2)} in the non-radial case, we have to assume an additional technical condition H3(k).  We believe it can be removed.
\end{rem}

To prove the above theorem we will resolve on the ideas in \cite{GLNW, Guo}.  The main new technical difficulty is that there is no scaling invariance for the propagator $e^{it\om(D)}$.  Consider an estimate of the form
\begin{align}\label{eq:pkest}
\norm{e^{it\om(D)}P_{k}\phi}_{L_t^qL_{x}^rL_\sigma^2}\les&
C(k)\norm{\phi}_{L_x^2}.
\end{align}
For \eqref{eq:pkest} it is equivalent to show
\begin{align}\label{eq:Stri2}
\norm{T_k f}_{L_t^qL_{x}^rL_\sigma^2}\les C(k)\norm{f}_{L_x^2},
\end{align}
where
\[T_k f(t,x)=\int_{\R^d}e^{i(x\xi+t\om(|\xi|))}\chi_k(|\xi|)f(\xi)d\xi.\]
We expand $f$ by the orthonormal basis $\{Y_m^l\}$, $m\geq
0,1\leq l\leq l(m)$ of spherical harmonics with
$l(m)=C_{d+m-1}^m-C_{d+m-3}^{m-2}$, such that
\[f(\xi)=f(\rho \sigma)=\sum_{m\geq 0}\sum_{1\leq l\leq l(m)}a_m^l(\rho)Y_m^l(\sigma), \quad \sigma\in \cir^{d-1}.\]
Then we get (by Lemma 3.10, Chapter IV, p158 in \cite{Stein1}) with $\nu=\nu(k)=\frac{d-2+2m}{2}$
\[T_kf(t,x)=\sum_{m,l}(2\pi)^{d/2}i^{-m}\widetilde T_k^\nu (a_m^l)(t,|x|)Y_m^l(x/|x|),\]
where
\begin{align*}
\widetilde T_k^\nu(h)(t,s)=&s^{-\frac{d-2}{2}}\int e^{it\om(\rho)}J_\nu(s\rho)\rho^{d/2}\chi_k(\rho)h(\rho)d\rho\\
=&s^{-\frac{d-2}{2}}2^{k}2^{kd/2}\int e^{it\om(2^k\rho)}J_\nu(2^ks\rho)\rho^{d/2}\chi_0(\rho)h(2^k\rho)d\rho
\end{align*}
with the Bessel function $J_\nu(r)$
\begin{align*}
J_\nu(r)=\frac{(r/2)^\nu}{\Gamma(\nu+1/2)\pi^{1/2}}
\sum_{m=0}^\infty \frac{(ir)^m}{m!}\int_{-1}^1t^m(1-t^2)^{\nu-1/2}dt, \quad \nu>-1/2.
\end{align*}
Thus \eqref{eq:Stri2}
becomes
\begin{align}\label{eq:Stri3}
\norm{\widetilde T_k^\nu (a_m^l)}_{L_t^q\Lr_{s}^rl_{m,l}^2}\les C(k)
\norm{\{a_m^l(s)\}}_{\Lr_s^2l_{m,l}^2}.
\end{align}
To prove \eqref{eq:Stri3}, since $q,r\geq 2$, it is equivalent to show
\begin{align}\label{eq:goal}
\norm{s^{\frac{d-1}{r}-\frac{d-2}{2}}T_k^\nu (h)}_{L_t^q L_{s}^r} \les 2^{-k}C(k)\norm{h}_{L^2},
\end{align}
with a bound independent of $\nu$, where $T_k^\nu$ is defined by
\begin{align}
T_k^\nu(h)(t,s)=& \int e^{it\om(2^k\rho)}J_\nu(2^ks\rho)\chi_0(\rho)h(\rho)d\rho
\end{align}
Note that in the radial case, to prove \eqref{eq:pkest} it suffices to prove \eqref{eq:goal} with $\nu=\frac{d-2}{2}$. In the general non-radial case 
we need uniform information of $J_\nu$ with respect to $\nu$.  We will need the uniform decay (see e.g \cite{Guo})
\begin{align}\label{eq:unidecay}
|J_\nu(r)|+|J_\nu'(r)|\les r^{-1/3}(1+ r^{-1/3}|r-\nu|)^{-1/4}
\end{align}
from which we get
\begin{align}\label{eq:BessL2}
\norm{J_\nu(r)}_{L^2_{r\sim R}}+\norm{J'_\nu(r)}_{L^2_{r\sim R}}\les 1, \quad \forall R>1.
\end{align}

For the region $|s|\leq 2^{-k}$, we get from the Taylor's expansion that
\begin{align*}
T_k^\nu h(s)=&\int_\R e^{it\om(2^k\ro)}\chi_0(\ro)h(\ro)\frac{(2^k\ro s/2)^\nu}{\Gamma(\nu+1/2)\pi^{1/2}}\\
&\qquad \times \sum_{m=0}^\infty \frac{(i2^k\ro s)^m}{m!}\left(\int_{-1}^1s^m(1-s^2)^{\nu-1/2}ds\right)
d\ro.
\end{align*}
If $\om$ satisfies H1(k), then by the Hausdorff-Young equality we have
\begin{align}\label{eq:j<-k}
\begin{split}
&\norm{\chi_{\leq -k} (s)s^{\frac{d-1}{r}-\frac{d-2}{2}}T_k^\nu (h)}_{L_t^qL_s^r}\\
\les& \sum_{m=0}^\infty \frac{C^m}{m!}2^{-k(\frac{d-1}{r}-\frac{d-2}{2})}2^{-k/r}\normo{\int_\R e^{it\om(2^k\ro)}\chi_0(\ro)\ro^{\nu+m}h(\ro)d\ro}_{L_t^q}\\
\les&2^{-k(\frac{d}{r}-\frac{d-2}{2})}\sup_{\ro\sim 1}|2^k\om'(2^k\ro)|^{-1/q}\norm{h}_2\\
\les&2^{-k}2^{k(\frac{d}{2}-\frac{d}{r})}2^{-k\al/q}\norm{h}_2.
\end{split}
\end{align}
It remains to deal with the region $|s|\geq 2^{-k}$. We decompose 
\[T_k^\nu (h)=\sum_{j\geq -k}T_{j,k}^\nu (h)\]
with
\begin{align}\label{eq:Tjknu}
T_{j,k}^\nu (h)=\chi_j(s)\int e^{it\om(2^k\rho)}J_\nu(2^ks\rho)\chi_0(\rho)h(\rho)d\rho.
\end{align}
We have the following simple estimates.
\begin{lem}\label{lem:H1}
Assume $k\in \Z$, $\om$ satisfies $\mathrm{H1(k)}$, $j\geq -k$ and $2\leq q\leq r\leq \infty$. Then 
\begin{align}
\norm{T_{j,k}^\nu h}_{L_t^qL_s^r}\les& 2^{-(j+k)(\frac{1}{2}-\frac{1}{q})}2^{-k/r}2^{-k\al/q}\norm{h}_2.
\end{align}
\end{lem}
\begin{proof}
For $q=2$, by Sobolev embedding, Plancherel's equality in $t$ and \eqref{eq:BessL2} we obtain
\begin{align*}
\norm{T_{j,k}^\nu h}_{L_t^2L_s^r}\les& 2^{-k/r}\normo{\chi_{j+k}(s)\int e^{it\om(2^k\rho)}J_\nu(s\rho)\chi_0(\rho)h(\rho)d\rho}_{L_t^2L_s^r}\\
\les&
2^{-k/r}(\norm{\chi_{j+k}(s)\int e^{it\om(2^k\rho)}J_\nu(s\rho)\chi_0(\rho)h(\rho)d\rho}_{L_t^2L_s^2}\\
&\qquad+\norm{\p_s[\chi_{j+k}(s)\int e^{it\om(2^k\rho)}J_\nu(s\rho)\chi_0(\rho)h(\rho)d\rho]}_{L_t^2L_s^2})\\
\les&2^{-k/r}2^{-k\al/2}
(\norm{J_\nu(s)}_{L_s^2}+\norm{J_\nu'(s)}_{L_s^2})\norm{h}_2\les 2^{-k/r}2^{-k\al/2}\norm{h}_2.
\end{align*}
For $q=r=\infty$, by \eqref{eq:BessL2} we have
\begin{align*}
\norm{T_{j,k}^\nu h}_{L_t^\infty L_s^\infty}\les& \normo{\int |\chi_{j+k}(s)J_\nu(s\rho)\chi_0(\rho)h(\rho)|d\rho}_{L_s^\infty}\\
\les& 2^{-(j+k)/2}\norm{h}_2.
\end{align*}
Therefore, the general $(q,r)$ estimate follows from interpolation between $(\infty, \infty)$ and $(2, \frac{2r}{q})$.
\end{proof}

With the above lemma we are ready to prove part (1) of Theorem \ref{thm}. If $d\geq 3$ and $\frac{2(d-1)}{d-2}<r\leq \infty$, we have
\begin{align*}
\norm{\chi_{\geq -k} (s)s^{\frac{d-1}{r}-\frac{d-2}{2}}T_k^\nu (h)}_{L_t^2L_s^r}\les& \sum_{j\geq -k}2^{j(\frac{d-1}{r}-\frac{d-2}{2})}\norm{T_{j,k}^\nu (h)}_{L_t^2L_s^r}\\
\les&2^{k(\frac{d-2}{2}-\frac{d-1}{r})}2^{-k/r}2^{-k\al/2}\norm{h}_2\\
\les&2^{-k}2^{k(\frac{d}{2}-\frac{d}{r})}2^{-k\al/2}\norm{h}_2.
\end{align*}
Combining the above inequality with \eqref{eq:j<-k}, we get the desired result.  If $d=2$, $2\leq q\leq r\leq \infty$ and $1/q<(d-1)(1/2-1/r)$, we have
\begin{align*}
\norm{\chi_{\geq -k} (s)s^{\frac{d-1}{r}-\frac{d-2}{2}}T_k^\nu (h)}_{L_t^qL_s^r}\les& \sum_{j\geq -k}2^{j(\frac{d-1}{r}-\frac{d-2}{2})}\norm{T_{j,k}^\nu (h)}_{L_t^qL_s^r}\\
\les&\sum_{j\geq -k}2^{j(\frac{d-1}{r}-\frac{d-2}{2})}2^{-(j+k)(\frac{1}{2}-\frac{1}{q})}2^{-k/r}2^{-k\al/q}\norm{h}_2\\
\les&2^{-k}2^{k(\frac{d}{2}-\frac{d}{r})}2^{-k\al/q}\norm{h}_2.
\end{align*}
Again with \eqref{eq:j<-k}, we get the desired result.  

In the rest of this section, we prove part (2) of Theorem \ref{thm}.  The uniform decay of the Bessel functions are not enough, but we also need to exploit the uniform oscillations. We will divide our proof into two cases: radial case and non-radial case.  In the non-radial case, we need an additional technical condition $\mathrm{H3(k)}$.

\subsection{Radial case}
This subsection is devoted to prove Theorem \ref{thm} (2) in the radial case.  In the previous subsection we only used the uniform decay of $J_\nu$.  We also need to exploit the oscillation.  In the radial case, we need to refine the estimate of $\norm{T_{j,k}^\nu}_{L^2\to L_t^2L_s^\infty}$ when $\nu=\frac{d-2}{2}$.
We have
\begin{align}\label{eq:Besselinfty}
J_{\frac{d-2}{2}}(r)=\frac{e^{i(r-\frac{(d-1)\pi}{4})}+e^{-i(r-\frac{(d-1)\pi}{4})}}{2r^{1/2}}+C_dr^{\frac{d-2}{2}}e^{-ir}E_+(r)-\wt C_dr^{\frac{d-2}{2}}e^{ir}E_-(r),
\end{align}
where $\p_r^kE_\pm(r)\les r^{-k-(d+1)/2}$ for $r\ge1$, $k=0,1,\cdots$, $C_d,\wt C_d$ are
constants, see \cite{Stein2}. 
We will need the Van der Corput lemma (see
\cite{Stein2}):

\begin{lem}[Van der Corput]\label{lem:staph}
Suppose $\phi$ is real-valued and smooth in $(a,b)$, and that
$|\phi^{(k)}(x)|\geq 1$ for all $x\in (a,b)$. Then
\[\aabs{\int_a^b e^{i\lambda \phi(x)}\psi(x)dx}\leq c_k \lambda^{-1/k}\bigg[|\psi(b)|+\int_a^b|\psi'(x)|dx\bigg]\]
and if $\psi'$ changes its sign $N$ times on $[a,b]$
\begin{align}
\aabs{\int_a^b e^{i\lambda \phi(x)}\psi(x)dx}\leq c_k N\lambda^{-1/k}\sup_{x\in [a,b]}|\psi(x)|
\end{align}
hold when (i) $k\geq 2$, or (ii) $k=1$ and $\phi'(x)$ is monotonic.
Here $c_k$ is a constant depending only on $k$. 
\end{lem}

\begin{lem}\label{lem:H2rad}
Assume $\om$ satisfies $\mathrm{H2(k)}$, $\nu=\frac{d-2}{2}$ and $j\geq -k$. Then for $2\leq r\leq \infty$
\begin{align}
\norm{T_{j,k}^\nu h}_{L_t^2L_s^r}\les&2^{-k\al/2}2^{-k/r}[2^{-(j+k)/4}2^{k(\al-\be)/4}]^{1-2/r}\norm{h}_2.
\end{align}
\end{lem}
\begin{proof}
The case $r=2$ was proved in Lemma \ref{lem:H1}. By interpolation it suffices to prove the case $r=\infty$. Inserting \eqref{eq:Besselinfty}
into \eqref{eq:Tjknu}, we then divide $T^\nu_{j,k}f$ into two
parts: the main term and the error term, namely
\begin{align}\label{eq:decfree}
T^\nu_{j,k}f=M_{j,k}f(t,2^ks)+E_{j,k}f(t,2^ks)
\end{align}
where by ignoring constant multiples
\begin{align*}
M_{j,k}f(s)=&\chi_{j+k}(s)\int e^{i(s\ro+t\om(2^k\ro))}(s\ro)^{-1/2}\chi_0(\ro)f(\ro)d\ro+c.c.\\
E_{j,k}f(s)=&\chi_{j+k}(s)\int e^{i(s\ro+t\om(2^k\ro))}(s\ro)^{(n-2)/2}E_+(s\ro)\chi_0(\ro)f(\ro)d\ro+c.c.
\end{align*}
Here $c.c.$ means similar terms with $s$ replaced by $-s$ which can be handled in the same way. 
First we consider $E_{j,k}f$.  By Sobolev embedding, Plancherel's equality in $t$, H1(k) and \eqref{eq:unidecay}, we get
\begin{align*}
\norm{E_{j,k}f}_{L_t^2L_s^\infty}\les \norm{E_{j,k}f}_{L_t^2L_s^2}+\norm{\partial_s E_{j,k}f}_{L_t^2L_s^2}\les 2^{-(j+k)}2^{-k\al/2}\norm{f}_2.
\end{align*}

Next we consider $M_{j,k}f$. By $TT^*$ argument, we have
\begin{align*}
M_{j,k}M_{j,k}^*F=&\iint \left(\int e^{i(t-s)\om(2^k\x)+i(x-y)\x}\xi^{-1}\chi_0^2(\x)d\xi\right)\\
&\quad\cdot |xy|^{-1/2}\chi_{j+k}(|x|)\chi_{j+k}(|y|)F(s,y)dyds,
 \end{align*}
and thus
\EQ{
 |M_{j,k}M_{j,k}^*F| \les 2^{-j-k}|(e^{it\om(2^kD)}\F^{-1}\chi_0^2)(x)\chi_{\leq j+k}(|x|)|*|F|.}
Let $K(t,x)=(e^{it\om(2^kD)}\F^{-1}\chi_0^2)(x)$. Then we get
\begin{align}
\norm{M_{j,k}}_{L_x^2\to L_t^2L_x^\infty}\les 2^{-(j+k)/2} \|K(t,x)\|_{L^1_{t\in\R}L^{\infty}_{|x|<2^{j+k}}}^{1/2}.
\end{align}
We have 
\begin{align*}
 K(t,x)=\int e^{it\om(2^k\ro)+ix\ro}\chi_0^2(\ro)d\ro.
\end{align*}
Let $\psi(\ro)=t\om(2^k\ro)+\ro x$ . Then
$\psi'(\rho)=t2^{k}\om'(2^k\ro)+x,
\psi''(\rho)=t2^{2k}\om''(2^k\rho)$. By H2(k) we have 
\EQ{
|\psi''(\ro)|\ges |t|2^{k\be}}in the support of $\chi_0$. 
Hence by Lemma \ref{lem:staph}
\begin{align*}
 |K(t,x)|\les |t|^{-1/2}2^{-k\be/2}.
\end{align*}
Moreover, we see that if $|t| 2^{k\al}\gg 2^{j+k}$, then $|\psi'(\rho)|\ges
|t| 2^{k\al}$, and thus using integration by parts twice and in view of the assumption H2(k) we get
\begin{align*}
|K(t,x)|\leq& \int \big|\partial_\rho[\psi'(\rho)^{-1}\partial_\rho
(\chi_0^2(\rho)\psi'(\rho)^{-1})]\big| d\rho\\
\les&\int \big|\psi'(\rho)^{-2}\big|+ |\psi'(\rho)^{-1}|\cdot \big|\p_\rho[\psi'(\rho)^{-1}]\big|+ |\psi'(\rho)^{-1}|\cdot\big|\p_\rho^2[\psi'(\rho)^{-1}]\big| d\rho\\
\les& \sup_\rho |\psi'(\rho)|^{-2}+\sup_\rho |\psi'(\rho)|^{-1}\sup_\rho |\p_{\rho}[\psi'(\rho)^{-1}]|\\
\les& |t|^{-2}2^{-2k\al},
\end{align*}
since $|\p_{\rho}[\psi'(\rho)^{-1}]|=\aabs{\frac{\psi''(\rho)}{\psi'(\rho)^2}}\les \frac{1}{|\psi'(\ro)|}$. 
Thus we get
\begin{align*}
|K(t,x)\chi_{j+k}(x)|\les & |t|^{-1/2}2^{-k\be/2}1_{\{|t| 2^{k\al}\les 2^{j+k}\}}+|t|^{-2}2^{-2k\al}1_{\{|t| 2^{k\al}\gg 2^{j+k}\}},
\end{align*}
and hence
\begin{align*}
\norm{K}_{L_t^1L^\infty_{|x|\les 2^{j+k}}}\les
2^{(j+k-k\al)/2}2^{-k\be/2}+2^{-(j+k)}2^{-k\al}.
\end{align*}
Thus, we get
\begin{align}
\norm{M_{j,k}}_{L_x^2\to L_t^2L_x^\infty}\les & 2^{-(j+k)/4}2^{-k(\al+\be)/4}+2^{-(j+k)}2^{-k\al/2}.
\end{align}
By the assumption we have and thus conclude the proof.
\end{proof}

Now we prove Theorem \ref{thm} (2) in the radial case.  By Lemma \ref{lem:H1} and Lemma \ref{lem:H2rad} we get for $\nu=\frac{d-2}{2}$ 
\begin{align}
\norm{T_{j,k}^\nu h}_{L_t^2L_s^r}\les&2^{-k\al/2}2^{-k/r}\min([2^{-(j+k)/4}2^{k(\al-\be)/4}]^{1-2/r},1)\norm{h}_2.
\end{align}
If $d\geq 2$ and $\frac{4d-2}{2d-3}<r<\frac{2d-2}{d-2}$, we have
\begin{align*}
&\norm{\chi_{\geq -k} (s)s^{\frac{d-1}{r}-\frac{d-2}{2}}T_k^\nu (h)}_{L_t^2L_s^r}\\
\les& \sum_{j\geq -k}2^{j(\frac{d-1}{r}-\frac{d-2}{2})}\norm{T_{j,k}^\nu (h)}_{L_t^2L_s^r}\\
\les&2^{-k/r}2^{-k\al/2}\sum_{j\geq -k}2^{j(\frac{d-1}{r}-\frac{d-2}{2})}\min([2^{-(j+k)/4}2^{k(\al-\be)/4}]^{1-2/r},1)\norm{h}_2\\
\les&2^{-k/r}2^{-k\al/2}\bigg(\sum_{j=-k}^{k[\al-\be-1]}2^{j(\frac{d-1}{r}-\frac{d-2}{2})}\\
&\qquad+\sum_{j=k[\al-\be-1]}^{\infty}2^{j(\frac{d-1}{r}-\frac{d-2}{2})}[2^{-(j+k)/4}2^{k(\al-\be)/4}]^{1-2/r}\bigg)\norm{h}_2\\
\les&2^{-k}2^{k(\frac{d}{2}-\frac{d}{r})}2^{-k\al/2}2^{k(\al-\be)(\frac{d-1}{r}-\frac{d-2}{2})}\norm{h}_2,
\end{align*}
and similarly for $r=\frac{2d-2}{d-2}$ we have
\begin{align}
\norm{\chi_{\geq -k} (s)s^{\frac{d-1}{r}-\frac{d-2}{2}}T_k^\nu (h)}_{L_t^2L_s^r}\les 2^{-k}\jb{k(\al-\beta)}2^{k(\frac{d}{2}-\frac{d}{r})}2^{-k\al/2}\norm{h}_2.
\end{align}
Therefore, by the above inequality, \eqref{eq:j<-k} and \eqref{eq:goal} we complete the proof.

\subsection{Non-radial case}
In the non-radial case we need to refine the estimate of $\norm{T_{j,k}^\nu}_{L^2\to L_t^2L_s^\infty}$ uniformly for $\nu\geq \frac{d-2}{2}$. We may assume $\nu\geq 11$ since the case $\nu\leq 10$ can be handled exactly as the radial case. Thus, we need to exploit uniformly oscillations and decay of $J_\nu$.  We will use
the Schl\"{a}fli's integral
representation of Bessel function (see p. 176, \cite{Watson}):
\begin{align}\label{eq:Besselint}
J_\nu(r)=&\frac{1}{2\pi}\int_{-\pi}^\pi e^{i(r\sin x-\nu
x)}dx-\frac{\sin(\nu\pi)}{\pi}\int_0^\infty
e^{-\nu\tau-r\sinh \tau}d\tau \nonumber\\
:=&J_\nu^M(r)-J_\nu^E(r).
\end{align}
We also need the following decay and asymptotical property. 

\begin{lem}[asymptotical property, Lemma 2.5 in \cite{Guo}]\label{lem:Bessel}
Let $\nu>10$ and $r>\nu+\nu^{1/3}$. Then

(1) We have
\[J_\nu(r)=\frac{1}{\sqrt{2\pi}}\frac{e^{i\theta(r)}+e^{-i\theta(r)}}{(r^2-\nu^2)^{1/4}}+h(\nu,r),\]
where $\theta(r)=(r^2-\nu^2)^{1/2}-\nu \arccos \frac{\nu}{r}-\frac \pi 4$,
and
\[|h(\nu,r)|\les \bigg(\frac{\nu^2}{(r^2-\nu^2)^{7/4}}+\frac{1}{r}\bigg)1_{[\nu+\nu^{1/3},2\nu]}(r)+r^{-1}1_{[2\nu,\infty)}(r).\]

(2) Let $x_0=\arccos \frac{\nu}{r}$. For any $K\in \N$ we have
\begin{align*}
h(\nu,r)=&(2\pi)^{-1/2}e^{i\theta(r)}x_0\sum_{k=1}^K\frac{
(rx_0^3)^{-k-1/2}a_k(x_0)}{k!}\\
&+(2\pi)^{-1/2}e^{-i\theta(r)}x_0\sum_{k=1}^K\frac{(rx_0^3)^{-k-1/2}\tilde
a_k(x_0)}{k!}+\tilde h(\nu,r)
\end{align*}
with functions $|\p^l a_k|+|\p^l \tilde a_k|\les 1$ for any $l\in
\N$ and
\[|\tilde h(\nu,r)|\les \bigg(\frac{r^{\frac{K}{2}+\frac{1}{4}}}{(r-\nu)^{\frac{3K}{2}+7/4}}+\frac{1}{r}\bigg)1_{[\nu+\nu^{1/3},2\nu]}(r)+r^{-1}1_{[2\nu,\infty)}(r).\]
Moreover, if $\nu\in \Z$, we have better estimate
\[|\tilde h(\nu,r)|\les \frac{r^{\frac{K}{2}+\frac{1}{4}}}{(r-\nu)^{\frac{3K}{2}+7/4}}1_{[\nu+\nu^{1/3},2\nu]}(r)+r^{-3/2}1_{[2\nu,\infty)}(r).\]
\end{lem}

Fixing $\lambda$ with $2^{\frac{j+k}{3}}\ll \lambda\les 2^{\frac{3(j+k)}{5}}$, we decompose 
\begin{align}
T_{j,k}^\nu (h)(t,s)=\sum_{l=1}^3 T_{j,k}^{\nu,l} (h)(t,2^ks),
\end{align}
where
\begin{align*}
T_{j,k}^{\nu,l} (h)(t,s)=&\chi_{j+k}(s)\int e^{it\om(2^k\rho)}J_\nu(s\rho)\gamma_l(\frac{s\ro-\nu}{\lambda})\chi_0(\rho)h(\rho)d\rho\end{align*}
with $\gamma_1(x)=\eta(x)$, $\gamma_2(x)=(1-\eta(x))1_{x<0}$, and
$\gamma_3(x)=(1-\eta(x))1_{x>0}$.  By the same argument  as in the proof of Lemma \ref{lem:H1} we can get
if $j\geq -k$ and $2\leq q\leq r\leq \infty$ then
\begin{align}\label{eq:Tjkvl}
\norm{T_{j,k}^{\nu,l} (h)}_{L_t^qL_s^r}\les& 2^{-(j+k)(\frac{1}{2}-\frac{1}{q})}2^{-k\al/q}\norm{h}_2.
\end{align}

\begin{lem}\label{lem:Tjkd3}
Assume $k\in \Z$, $\om$ satisfies $\mathrm{H2(k),\, H3(k)}$, $j\geq -k$ and $R=2^{j+k}$. Then for $2\leq r\leq \infty$
\begin{align*}
\norm{T_{j,k}^{\nu,1}  (h)}_{L_t^2L_s^r}\les& 2^{-k\al/2}\lambda^{1/4}R^{-1/4}\norm{h}_2,\\
\norm{T_{j,k}^{\nu,2}  (h)}_{L_t^2L_s^r}\les& 2^{-k\al/2}(\lambda^{-1}R^{1/4}2^{k(\al-\be)/4})^{1-2/r}\norm{h}_2,\\
\norm{T_{j,k}^{\nu,3}  (h)}_{L_t^2L_s^r}\les& 2^{-k\al/2}[(R^{-1/4}2^{k(\al-\be)/4})^{1-2/r}+R^{1/8}\lambda^{-5/8}]\norm{h}_2.
\end{align*}
\end{lem}
\begin{proof}
For $T_{j,k}^{\nu,1}, T_{j,k}^{\nu,2}$, we follow closely the ideas in \cite{GLNW, Guo}. For $T_{j,k}^{\nu,3}$ we use a new argument to handle the error term. By interpolation, we only need to show the estimates for
$r=2,\infty$. 

{\bf Step 1:} estimate of $T_{j,k}^{\nu,1}$.

As the proof of Lemma \ref{lem:H1}, by Sobolev embedding and Plancherel's equality in $t$ we obtain
\begin{align*}
\norm{T_{j,k}^{\nu,1} (h)}_{L_t^2L_s^r}
\les&2^{-k\al/2}
(\norm{J_\nu(s)\ga_1(\frac{s-\nu}{\lambda})}_{L_{s\sim R}^2}+\norm{J'_\nu(s)\ga_1(\frac{s-\nu}{\lambda})}_{L_{s\sim R}^2})\norm{h}_2\\
\les& \lambda^{1/4}R^{-1/4}2^{-k\al/2}\norm{h}_2.
\end{align*}
So we get the desired estimate.

{\bf Step 2:} estimate of $T_{j,k}^{\nu,2}$.

The estimate for $r=2$ is given by \eqref{eq:Tjkvl} and hence we only need to consider the case $r=\infty$.
We have $\nu>s\rho+\lambda$ in the
support of $\gamma_2(\frac{s\rho-\nu}{\lambda})$. Thus we use the
formula \eqref{eq:Besselint}. Without loss of generality, we assume
$J_\nu^M=\frac{1}{\pi}\int_0^{\pi}
 e^{i(r\rho\sin \theta-\nu
\theta)}d\theta$ (its conjugate part can be handled in the same
way), and decompose
\[T_{j,k}^{\nu,2}(h):=M_{j,k}^{\nu,2}(h)(s)+E_{j,k}^{\nu,2}(h)(s)\]
where
\begin{align*}
M_{j,k}^{\nu,2} (h)=&\chi_{j+k}(s)\int
e^{-it\om(2^k\rho)}\bigg(\int_0^{\pi/2}
 e^{i(s\rho\sin \theta-\nu
\theta)}d\theta\bigg) \gamma_2(\frac{s\rho-\nu}{\lambda})\chi_0(\rho)h(\rho)d\rho,\\
E_{j,k}^{\nu,2} (h)=&\chi_{j+k}(s)\int
e^{-it\om(2^k\rho)}\bigg(J_\nu^E(s\rho)+\int_{\pi/2}^{\pi}
 e^{i(s\rho\sin \theta-\nu
\theta)}d\theta\bigg)\gamma_2(\frac{s\rho-\nu}{\lambda})\chi_0(\rho)h(\rho)d\rho.
\end{align*}
By the obvious decay estimate
\begin{align*}
&|J_\nu^E(s\rho)|+|\partial_s[J_\nu^E(s\rho)]|+\left|\int_{\pi/2}^{\pi}
 e^{i(s\rho\sin \theta-\nu
\theta)}d\theta \right|+\left|\partial_s\left[\int_{\pi/2}^{\pi}
 e^{i(s\rho\sin \theta-\nu
\theta)}d\theta\right] \right|\\
\les& (\nu+s)^{-1}, \quad \mbox{ for } s\ges 1, \ro\sim 1,
\end{align*}
and by the similar proof of Lemma \ref{lem:H1} we can get that 
for $2\leq r\leq \infty$
\begin{align}
\norm{E_{j,k}^{\nu,2}(h)}_{L_t^2 L_{s}^r}\les 2^{-k\al/2}R^{-1/2}\norm{h}_2,
\end{align}
which is acceptable. It remains to prove
\begin{align}
\norm{M_{j,k}^{\nu,2} (h)}_{L_t^2L_{s}^\infty}\les 2^{-k(\al+\be)/4}\lambda^{-1}R^{1/4}
\norm{h}_{L^2}.
\end{align}
Denote
$\phi(s,\rho,\theta)=s\rho\sin \theta-\nu \theta$. Integrating by
part, we can decompose further
\begin{align*}
M_{j,k}^{\nu,2} (h)=&\chi_{j+k}(s)\int\bigg(\frac
{e^{i\phi(s,\rho,\theta)}}{i(s\rho
\cos\theta-\nu)}\bigg|_{\theta=\pi/2}-\frac
{e^{i\phi(s,\rho,\theta)}}{i(s\rho
\cos\theta-\nu)}\bigg|_{\theta=0}\\
&-\int_{0}^{\pi/2}
 e^{i(s\rho\sin \theta-\nu
\theta)}\frac{s\rho \sin\theta}{(s\rho \cos
\theta-\nu)^2}d\theta\bigg)\gamma_2(\frac{s\rho-\nu}{\lambda})e^{-it\om(2^k\rho)}\chi_0(\rho)h(\rho)d\rho\\
:=&M_{j,k,1}^{\nu,2}(h)-M_{j,k,2}^{\nu,2}(h)-M_{j,k,3}^{\nu,2}(h).
\end{align*}
We have
\begin{align*}
M_{j,k,1}^{\nu,2}(h)=&\chi_{j+k}(s)\int\frac {e^{i(s\rho-\nu
\pi/2)}}{-i\nu}\gamma_2(\frac{s\rho-\nu}{\lambda})e^{-it\om(2^k\rho)}\chi_0(\rho)h(\rho)d\rho,\\
M_{j,k,2}^{\nu,2}(h)=&\chi_{j+k}(s)\int\frac
{1}{i(s\rho-\nu)}\gamma_2(\frac{s\rho-\nu}{\lambda})e^{-it\om(2^k\rho)}\chi_0(\rho)h(\rho)d\rho.
\end{align*}
For $M_{j,k,2}^{\nu,2}$,  by $TT^*$ argument, it is equivalent to
show
\begin{align}
\norm{M_{j,k,2}^{\nu,2}(M_{j,k,2}^{\nu,2})^*f}_{L_t^2L_{r}^\infty}\les 2^{-k(\al+\be)/2}\lambda^{-2}R^{1/2}\norm{f}_{L_t^2L_{r}^1}
\end{align}
where
\[M_{j,k,2}^{\nu,2}(M_{j,k,2}^{\nu,2})^*f=\int K_2(t-t',r,r')f(t',r')dt'dr'\]
with the kernel
\[K_2(t,r,r')=\int e^{-it\om(2^k\rho)}\frac
{\chi_0\big(\frac
rR\big)}{r\rho-\nu}\gamma_2(\frac{r\rho-\nu}{\lambda})\frac
{\chi_0\big(\frac
{r'}R\big)}{r'\rho-\nu}\gamma_2(\frac{r'\rho-\nu}{\lambda})\chi_0^2(\rho)d\rho.\]
It suffices to prove
\begin{align}\label{eq:K2est}
\norm{K_2}_{L_t^1L^\infty_{r,r'}}\les 2^{-k(\al+\be)/2}\lambda^{-2}R^{1/2}.
\end{align}
Denote $F_2(\ro)=\frac {\chi_0\big(\frac
rR\big)}{i(r\rho-\nu)}\gamma_2(\frac{r\rho-\nu}{\lambda})\frac
{\chi_0\big(\frac
{r'}R\big)}{i(r'\rho-\nu)}\gamma_2(\frac{r'\rho-\nu}{\lambda})\chi_0^2(\rho)$. Then $F_2$ and $F_2'$ are both piecewise monotone on $[0,\infty)$, and we have $|F_2|\les \lambda^{-2}, |\p_\rho F_2|\les \lambda^{-3}R$.
By Lemma \ref{lem:staph} we get
\begin{align*}
|K_2(t,r,r')|\les& 2^{-k\be/2}|t|^{-1/2} \int |\partial_\rho
F_2|d\rho\les 2^{-k\be/2}\lambda^{-2}|t|^{-1/2}.
\end{align*}
On the other hand, using  twice integration by part, we
get
\begin{align*}
|K_2|\les& \int
|t|^{-2}\aabs{\partial_\rho\big[2^{-k}\om'(2^k\ro)^{-1}\partial_\rho[2^{-k}\om'(2^k\ro)^{-1}F_2]\big]}d\rho\\
\les& \int
|t|^{-2}2^{-2k}\big(|\om'(2^k\ro)|^{-2}|\p_\rho^2F_2|+\aabs{\p_\rho[\om'(2^k\ro)^{-1}]\om'(2^k\ro)^{-1}\p_\rho F_2}\\
&\qquad\qquad\qquad\qquad+\aabs{\p_\rho^2[\om'(2^k\ro)^{-1}]\om'(2^k\ro)^{-1} F_2}\big)   d\rho\\
\les&
2^{-2k\al}|t|^{-2}R\lambda^{-3}.
\end{align*}
Then eventually we have
\[|K_2|\les 2^{-k\be/2} \lambda^{-2}|t|^{-1/2}1_{|t|\les 2^{-k\al}R}+|t|^{-2}\lambda^{-3}R2^{-2k\al}1_{|t|\gg 2^{-k\al}R}\]
which implies
\[\norm{K_2}_{L_t^1L^\infty_{r,r'}}\les 2^{-k(\al+\be)/2}\lambda^{-2}R^{1/2}.\]
For $M_{j,k,1}^{\nu,2}$, we have the same bound as $M_{j,k,2}^{\nu,2}$ by the similar and easier proof
since $\nu\ges R$. 

Now we consider $M_{j,k,3}^{\nu,2}$. Similarly, the kernel of
$M_{j,k,3}^{\nu,2}(M_{j,k,3}^{\nu,2})^*$ is
\begin{align*}
K_3(t-t',r,r')=&\int_{0}^{\pi/2}\int_{0}^{\pi/2}\int
e^{-i(t-t')\om(2^k\rho)}
 e^{i(r\rho\sin \theta-\nu
\theta)}\frac{\chi_0\big(\frac rR\big)r\rho \sin\theta}{(r\rho \cos
\theta-\nu)^2}
\gamma_2(\frac{r\rho-\nu}{\lambda})\\
&\times
 e^{-i(r'\rho\sin \theta'-\nu
\theta')}\frac{\chi_0\big(\frac {r'}R\big)r'\rho
\sin\theta'}{(r'\rho \cos
\theta'-\nu)^2}\gamma_2(\frac{r'\rho-\nu}{\lambda})\chi_0^2(\rho)d\rho
d\theta d\theta'.
\end{align*}
It suffices to prove
\[\norm{K_3}_{L_t^1L^\infty_{r,r'}}\les 2^{-k(\al+\be)/2}\lambda^{-2}R^{1/2}.\]
Denote
\[ 
\tau(\theta,\rho):=r\ro\cos\te-\nu, \pq \tau'(\theta',\rho):=r'\ro\cos\te'-\nu.
\]
Let $F_3(\theta,\theta',\rho)=\chi_0\big(\frac rR\big)\chi_0\big(\frac
{r'}R\big)\gamma_2(\frac{r\rho-\nu}{\lambda})\gamma_2(\frac{r'\rho-\nu}{\lambda})\chi_0^2(\rho)$.
By these notations we get
\begin{align*}
K_3(t,r,r')=&\int_{0}^{\pi/2}\int_{0}^{\pi/2}\int
e^{-i[t\om(2^k\rho)+\tau(\te,\rho)-\tau'(\te',\rho)]}
 F_3\p_\theta [\tau^{-1}]\p_{\theta'} [{\tau'}^{-1}]d\rho
d\theta d\theta'.
\end{align*}
It's easy to see that $F_3$ and $\p_\rho F_3$ are both piece-wise monotone,  $\p_\ro\p_\te (\tau^{-1}), \p_{\te}
({\tau}^{-1})$ (similarly for $\tau'^{-1}$) do not change the sign. Since the phase $|\p_\ro^2[t\om(2^k\rho)+\tau(\te,\rho)-\tau'(\te',\rho)]|\ges |t|2^{k\beta}$, we get by Lemma \ref{lem:staph} that
\begin{align*}
|K_3|\les& 2^{-k\be/2}|t|^{-1/2}\int_{0}^{\pi/2}\int_{0}^{\pi/2}\int
|\p_\ro (F_3\p_\theta [\tau^{-1}]\p_{\theta'} [{\tau'}^{-1}])|d\rho
d\theta d\theta'\\
\les & 2^{-k\be/2}|t|^{-1/2}\sup_{\ro,\te,\te'}|F_3\tau^{-1}\tau'^{-1}|\les 2^{-k\be/2}|t|^{-1/2}\lambda^{-2},
\end{align*}
On the other hand, for $2^{k\al}|t|\gg
R$, denoting $\phi_1=-2^kt\om'(2^k\rho)+r\sin\theta-r'\sin \theta'$, we get $|\phi_1|\ges |t|2^{k\alpha}$ and thus get
\begin{align*}
|K_3|\les &\int_{0}^{\pi/2}\int_{0}^{\pi/2}\int
|\partial_\rho\big(\phi_1^{-1}\partial_\rho[\phi_1^{-1}F_3\p_\theta [h^{-1}]\p_{\theta'} [{h'}^{-1}]]\big)|d\rho
d\theta d\theta'\\
\les& 2^{-2k\al}|t|^{-2}\lambda^{-3}R.
\end{align*}
Then we get
\[|K_3|\leq 2^{-k\be/2} \lambda^{-2}|t|^{-1/2}1_{|t|\les 2^{-k\al}R}+|t|^{-2}\lambda^{-3}R2^{-2k\al}1_{|t|\gg 2^{-k\al}R}\]
which implies that $\norm{K_3}_{L_t^1L_{r,r'}^\infty}\les
\lambda^{-2}R^{1/2}2^{-k(\al+\be)/2}$ as desired.

{\bf Step 3:} estimate of $T_{j,k}^{\nu,3}$.

The estimate for $r=2$ is given by \eqref{eq:Tjkvl} and hence we only need to consider the case $r=\infty$.
By the support of $\gamma_3$, we have
$r\rho>\nu+\lambda>\nu+\nu^{1/3}$ in the support of
$\gamma_3(\frac{r\rho-\nu}{\lambda})$. Thus we use the Lemma
\ref{lem:Bessel}, and decompose
\[T_{j,k}^{\nu,3}(h):=M_{j,k}^{\nu,3}(h)+E_{j,k}^{\nu,3}(h)\]
where
\begin{align*}
M_{j,k}^{\nu,3}(h)=&\chi_0\big(\frac rR\big)\int
e^{-it\om(2^k\rho)}\frac{e^{i\theta(r\rho)}+e^{-i\theta(r\rho)}}
{2 \sqrt{2\pi}(r^2\rho^2-\nu^2)^{1/4}}\gamma_3(\frac{r\rho-\nu}{\lambda})\chi_0(\rho)h(\rho)d\rho,\\
E_{j,k}^{\nu,3}(h)=&\chi_0\big(\frac rR\big)\int
e^{-it\om(2^k\rho)}h(\nu,r\rho)\gamma_3(\frac{r\rho-\nu}{\lambda})\chi_0(\rho)h(\rho)d\rho,
\end{align*}
with $\theta(r),h(\nu,r)$ given in Lemma \ref{lem:Bessel}.

First, we consider $M_{j,k}^{\nu,3}$. We may assume
\[M_{j,k}^{\nu,3} (h)=\chi_0\big(\frac rR\big)\int e^{-it\om(2^k\rho)}\frac{e^{i\theta(r\rho)}}
{(r^2\rho^2-\nu^2)^{1/4}}\gamma_3(\frac{r\rho-\nu}{\lambda})\chi_0(\rho)h(\rho)d\rho,\]
since the other term is similar. 
Let $\gamma(x)=\chi_0(x)\cdot 1_{x>0}$. We decompose further
$M_{j,k}^{\nu,3}(h)=\sum_{l: \lambda\leq 2^l \les
R}M_{j,k,l}^{\nu,3}(h)$ where
\begin{align}\label{eq:MR3k}
M_{j,k,l}^{\nu,3}(h)=\chi_0\big(\frac rR\big)\int
e^{-it\om(2^k\rho)}\frac{e^{i\theta(r\rho)}}
{(r^2\rho^2-\nu^2)^{1/4}}\gamma(\frac{r\rho-\nu}{2^l})\chi_0(\rho)h(\rho)d\rho.
\end{align}
It suffices to prove
\begin{align*}
\norm{M_{j,k,l}^{\nu,3}(h)}_{L_t^2L_r^\infty}\les 2^{-k(\al+\beta)/4}
(2^{l/8}R^{-3/8}+2^{-l/8}R^{-1/4})\norm{h}_2.
\end{align*}
By $TT^*$ argument, it suffices to prove
\begin{align}
\norm{M_{j,k,l}^{\nu,3}(M_{j,k,l}^{\nu,3})^*(f)}_{L_t^2L_r^\infty}\les
2^{-k(\al+\beta)/2}
(2^{l/4}R^{-3/4}+2^{-l/4}R^{-1/2})\norm{f}_{L_t^2L_r^1}.
\end{align}
The kernel for $M_{j,k,l}^{\nu,3}(M_{j,k,l}^{\nu,3})^*$ is
\[K(t-t',r,r')=\int e^{-i[(t-t')\om(2^k\rho)-\theta(r\rho)+\theta(r'\rho)]}\frac{\chi_0\big(\frac rR\big)\gamma(\frac{r\rho-\nu}{2^l})}
{(r^2\rho^2-\nu^2)^{1/4}}\frac{\chi_0\big(\frac
{r'}R\big)\gamma(\frac{r'\rho-\nu}{2^l})}
{(r'^2\rho^2-\nu^2)^{1/4}}\chi_0^2(\rho)d\rho.\]
It suffices to prove
\[\norm{K}_{L_t^1L^\infty_{r,r'}}\les 
2^{-k(\al+\beta)/2}
(2^{l/4}R^{-3/4}+2^{-l/4}R^{-1/2}).\]
Recall
$\theta(r)=(r^2-\nu^2)^{1/2}-\nu\arccos\frac{\nu}{r}-\frac{\pi}{4}$,
then direct computation shows
\begin{align*}
\theta'(r)=&(r^2-\nu^2)^{1/2}r^{-1},\\
\theta''(r)=&(r^2-\nu^2)^{-1/2}\nu^2r^{-2},\\
\theta'''(r)=&(r^2-\nu^2)^{-3/2}\frac{\nu^2}{r}(-3+\frac{2\nu^2}{r^2}).
\end{align*}
Denoting $G=\frac{\chi_0\big(\frac
rR\big)\gamma(\frac{r\rho-\nu}{2^{l}})}
{(r^2\rho^2-\nu^2)^{1/4}}\frac{\chi_0\big(\frac
{r'}R\big)\gamma(\frac{r'\rho-\nu}{2^l})}
{(r'^2\rho^2-\nu^2)^{1/4}}\chi_0^2(\rho)$,
$\phi_2=t\om(2^k\rho)-\theta(r\rho)+\theta(r'\rho)$. Then
\begin{align*}
\partial_\rho(\phi_2)=&2^kt\om'(2^k\ro)+
\frac{\rho(r'^2-r^2)}{\sqrt{r'^2\rho^2-\nu^2}+\sqrt{r^2\rho^2-\nu^2}}\\
\partial^2_\rho(\phi_2)=&2^{2k}t\om''(2^k\ro)-
\frac{(r'^2-r^2)}{\sqrt{r'^2\rho^2-\nu^2}+\sqrt{r^2\rho^2-\nu^2}}
\frac{\nu^2}{\sqrt{r'^2\rho^2-\nu^2}\sqrt{r^2\rho^2-\nu^2}}\\
\partial^3_\rho(\phi_2)=&2^{3k}t\om'''(2^k\ro)-\theta'''(r\rho)r^3+\theta'''(r'\rho)r'^3
\end{align*}
The key observation here is that $\partial_\rho(\phi_2)$ and $\partial^2_\rho(\phi_2)$ can not be both small. Hence we can apply the stationary phase method. 

{\it Case 1. } $R\sim \nu$

If $|\partial_\rho(\phi_2)|\ll
2^{k\al}|t|$, then $
\frac{|r'^2-r^2|}{\sqrt{r'^2\rho^2-\nu^2}+\sqrt{r^2\rho^2-\nu^2}}\ges 2^{k\al}|t|$, and $|\partial^2_\rho(\phi_2)|\ges 2^{\al k}|t|R2^{-l}$ on the support
of $G$ by the assumption $\text{H2(k) and H3(k)}$. Note that on the support of $G$, one has
\[\frac{|r'^2-r^2|}{\sqrt{r'^2\rho^2-\nu^2}+\sqrt{r^2\rho^2-\nu^2}}\les \sqrt{r'^2\rho^2-\nu^2}+\sqrt{r^2\rho^2-\nu^2}\les
2^{l/2}R^{1/2}.\] 
If $2^{k\al}|t|\les 2^{l/2}R^{1/2}$, we divide $K$
\[K=\int e^{-i\phi_2}G \eta_0(\frac{100\partial_\rho(\phi_2)}{2^{k\al}t})d\rho+
\int e^{-i\phi_2}G
[1-\eta_0(\frac{100\partial_\rho(\phi_2)}{2^{k\al}t})]d\rho:=I_1+I_2.\] By
Lemma \ref{lem:staph}, we obtain
\begin{align*}
|I_1|\les& 2^{-k\al/2}|t|^{-1/2}R^{-1/2}2^{l/2}\bigg(\int \abs{\partial_\rho [G
\eta_0(\frac{100\partial_\rho(\phi_2)}{t2^{k\al}})]}d\rho\bigg)\\
\les&2^{-k\al/2}|t|^{-1/2}R^{-1/2}2^{l/2}\bigg(\sup_{\rho}|G|\int |\p_\ro \eta_0(\frac{100\partial_\rho(\phi_2)}{t2^{k\al}})|d\ro+\int |\partial_\rho G|d\rho \bigg)\\
\les& 2^{-k\al/2}|t|^{-1/2}R^{-1/2}2^{l/2}2^{-l/2}R^{-1/2}\\
\les& 2^{-k\al/2}|t|^{-1/2}R^{-1}
\end{align*}
where we used the fact that $\eta_0', \p_\ro G$ change sign for finite times. 
For $I_2$, without loss of generality, we assume $r^2-r'^2>0$. Then
integrating by part, we get
\begin{align*}
|I_2|\les& \int \aabs{\partial_\rho \bigg((\partial_\rho\phi_2)^{-1}G[1-\eta_0(\frac{100\partial_\rho(\phi_2)}{t2^{k\al}})]\bigg)}d\rho\\
\les&2^{-k\al}|t|^{-1}2^{-l/2}R^{-1/2}. 
\end{align*}
Interpolating with the trivial estimate $|I_2|\les 2^{-l/2}R^{-1/2}$, we get
\[|I_2|\les \min(2^{-k\al}|t|^{-1},1)2^{-l/2}R^{-1/2}.\]

If $2^{k\al}|t|\gg 2^{l/2}R^{1/2}$, we have $|\partial_\rho(\phi_2)|\gg
2^{k\al}|t|$. Thus integrating by part, we get
\begin{align*}
|K|\les& \int \aabs{\partial_\rho \big[(\partial_\rho\phi_2)^{-1}\partial_\rho \big((\partial_\rho\phi_2)^{-1}G\big)\big]}d\rho\\
\les& \int
\aabs{(\partial_\rho\phi_2)^{-3}\partial^3_\rho\phi_2G}d\rho+\int
\aabs{(\partial_\rho\phi_2)^{-2}\partial_\rho^2G}d\rho\\
&+\int
\aabs{(\partial_\rho\phi_2)^{-3}\partial^2_\rho\phi_2\partial_\rho
G}d\rho+\int
\aabs{(\partial_\rho\phi_2)^{-4}(\partial^2_\rho\phi_2)^2 G}d\rho\\
:=&II_1+II_2+II_3+II_4.
\end{align*}
As for $I_2$, we can obtain
\[II_2+II_3+II_4\les 2^{-2k\al}|t|^{-2}2^{-l/2}R^{-1/2}R^22^{-2l}\les 2^{-2k\al}|t|^{-2}2^{-5l/2}R^{3/2}.\]
For $II_1$, we have
\begin{align*}
II_1\les& 2^{-3k\al}|t|^{-3}\lambda^{-1/2}R^{-1/2}\int
|-\theta'''(r\rho)r^3-\theta'''(r'\rho)r'^3|
\gamma(\frac{r\rho-\nu}{2^l})\gamma(\frac{r'\rho-\nu}{2^l})d\rho\\
&+2^{-2k\al}|t|^{-2}2^{-l/2}R^{-1/2}\int
\frac{|2^{3k}t\om'''(2^k\ro)|}{|2^kt\om'(2^k\ro)|}
\gamma(\frac{r\rho-\nu}{2^l})\gamma(\frac{r'\rho-\nu}{2^l})d\rho \\
\les&2^{-3k\al} |t|^{-3}2^{-l/2}R^{-1/2}
\sup_{\rho}\theta''(r\rho)r^2\\
&+2^{-2k\al}|t|^{-2}2^{-l/2}R^{-1/2}\int
\aabs{\p_\ro\brk{\frac{2^{2k}t\om''(2^k\ro)}{2^kt\om'(2^k\ro)}}}+\aabs{\brk{\frac{2^{2k}t\om''(2^k\ro)}{2^kt\om'(2^k\ro)}}^2}d\rho \\\\
\les& 2^{-3k\al} |t|^{-3}2^{-l}R+2^{-2k\al}|t|^{-2}2^{-l/2}R^{-1/2}.
\end{align*}
Thus, eventually we get
\begin{align*}
|K|\les& (2^{-k\al/2}|t|^{-1/2}R^{-1}+\min(2^{-k\al}|t|^{-1},1)2^{-l/2}R^{-1/2})1_{2^{k\al}|t|\les R^{1/2}2^{l/2}}\\
&+(2^{-2k\al}|t|^{-2}2^{-5l/2}R^{3/2}+2^{-3k\al} |t|^{-3}2^{-l}R)1_{2^{k\al}|t|\gg R^{1/2}2^{l/2}}
\end{align*}
which implies $\norm{K}_{L_t^1L_{r,r'}^\infty}\les 2^{-k\al}(2^{l/4}R^{-3/4}+2^{-l/4}R^{-1/2})$
as desired since $2^l\ges R^{1/3}$.

{\it Case 2. } $R\gg \nu$

In this case we may assume $2^{l}\sim R$ since $|r\ro-\nu|\sim R$. We observe that if $|\partial_\rho(\phi_2)|\ll
2^{k\al}|t|$, then $|\partial^2_\rho(\phi_2)|\ges 2^{\beta k}|t|$ on the support
of $G$ by the assumption $\text{H2(k) and H3(k)}$. Note that $|G|\les R^{-1}$.
Then as in Case 1, we can get
\begin{align*}
|I_1|\les& 2^{-k\beta/2}|t|^{-1/2}\bigg(\int |\partial_\rho [G
\eta_0(\frac{100\partial_\rho(\phi_2)}{t})]|d\rho\bigg)\\
\les& 2^{-k\beta/2}|t|^{-1/2}R^{-1}.
\end{align*}
The rest estimates are the same as Case 1. So we get
\begin{align*}
|K|\les& (2^{-k\beta/2}|t|^{-1/2}R^{-1}+\min(2^{-k\al}|t|^{-1},1)2^{-l/2}R^{-1/2})1_{2^{k\al}|t|\les R^{1/2}2^{l/2}}\\
&+(2^{-2k\al}|t|^{-2}2^{-5l/2}R^{3/2}+2^{-3k\al} |t|^{-3}2^{-l}R)1_{2^{k\al}|t|\gg R^{1/2}2^{l/2}}
\end{align*}
which implies $\norm{K}_{L_t^1L_{r,r'}^\infty}\les 2^{-k(\al+\beta)/2}R^{-1/2}$
as desired.

It remains to bound $E_{j,k}^{\nu,3}$. First, using the decay estimate of $h(\nu,r)$, we get
\begin{align}\label{eq:ER3L2}
\norm{E_{j,k}^{\nu,3}(f)}_{L_t^2L_r^2}\les
2^{-k\al/2}(\lambda^{-5/4}R^{1/4}+R^{-1/2})\norm{f}_{L^2}.
\end{align}
Since we do not have estimate on $\p_r h$, we can not get $L_t^2L_r^\infty$ estimate by Sobolev embedding as before. We need a different argument. The argument used here is also different from the previous works \cite{GLNW,Guo}.
We claim that 
\begin{align}\label{eq:ER3infty}
\norm{E_{j,k}^{\nu,3}(f)}_{L_t^2L_r^\infty}\les
2^{-k\al/2}R^{1/8}\lambda^{-5/8}\norm{f}_{L^2}.
\end{align}
By $TT^*$ argument \eqref{eq:ER3infty} is equivalent to 
\begin{align}
\norm{E_{j,k}^{\nu,3}(E_{j,k}^{\nu,3})^*(f)}_{L_t^2L_r^\infty}\les
2^{-k\al}R^{1/4}\lambda^{-5/4}\norm{f}_{L_t^2L_r^1}.
\end{align}
The kernel for $E_{j,k}^{\nu,3}(E_{j,k}^{\nu,3})^*$ is
\begin{align*}
K_E(t-t',r,r')=&\int e^{-i[(t-t')\om(2^k\rho)]}h(\nu,r\ro)h(\nu,r'\ro)\chi_0\big(\frac rR\big)\\
&\qquad \cdot\gamma_3(\frac{r\rho-\nu}{\lambda})
{\chi_0\big(\frac
{r'}R\big)\gamma_3(\frac{r'\rho-\nu}{\lambda})}
\chi_0^2(\rho)d\rho.
\end{align*}
It suffices to prove 
\[\norm{K_E}_{L_t^1L^\infty_{r,r'}}\les 2^{-k\al}R^{1/4}\lambda^{-5/4}.\]

By the decay estimate of $h$ given in Lemma \ref{lem:Bessel}, we have the trivial estimate
\begin{align*}
|K_E|\les& \norm{h(\nu,r\ro)\chi_0(\ro)\gamma_3(\frac{r\rho-\nu}{\lambda})}_{L_\ro^2}\cdot\norm{h(\nu,r'\ro)\chi_0(\ro)\gamma_3(\frac{r'\rho-\nu}{\lambda})}_{L_\ro^2}\\
\les& R^{-1/2}\lambda^{-5/2}+R^{-2}.
\end{align*}
On the other hand, we have
\[K_E=K_T-K_M\]
where $K_M=\sum_l K$ is the kernel for $M_{j,k}^{\nu,3}(M_{j,k}^{\nu,3})^*$, and
\begin{align*}
K_T=&\int e^{-i[t\om(2^k\rho)]}J_\nu(r\ro)J_\nu(r'\ro)\chi_0\big(\frac rR\big)\\
&\qquad \cdot\gamma_3(\frac{r\rho-\nu}{\lambda})
{\chi_0\big(\frac
{r'}R\big)\gamma_3(\frac{r'\rho-\nu}{\lambda})}
\chi_0^2(\rho)d\rho.
\end{align*}
If $2^{\al k}|t|\gg R$, by the estimates for $K$ we have
\[|K_M|\les 2^{-2k\al}|t|^{-2}\lambda^{-5/2}R^{3/2}.\]
Now we estimate $K_T$ in the range $2^{\al k}|t|\gg R$.  We will use \eqref{eq:BessL2} and the property of Bessel function
\[\p_z J_\nu(z)=\nu z^{-1}J_\nu-J_{\nu+1}.\]
Denote $G_E=J_\nu(r\ro)J_\nu(r'\ro)\chi_0\big(\frac rR\big)\gamma_3(\frac{r\rho-\nu}{\lambda})
{\chi_0\big(\frac
{r'}R\big)\gamma_3(\frac{r'\rho-\nu}{\lambda})}
\chi_0^2(\rho)$. We have $|J_\nu'|\les R^{-1/4}\lambda^{-1/4}$ in the support of $G_E$.
Then for $2^{\al k}|t|\gg R$ we have
\begin{align*}
|K_T|\les& \int \aabs{\p_\ro\bigg([2^kt\om'(2^k\ro)]^{-1} \p_\ro\big([2^kt\om'(2^k\ro)]^{-1}G_E\big)\bigg)}d\ro\\
\les&  2^{-2k\al}|t|^{-2}R^2\norm{J''_\nu(r\ro)\chi_0(\ro)}_{L_\ro^2}\cdot\norm{J_\nu(r'\ro)\chi_0(\ro)}_{L_\ro^2}\\
\les& 2^{-2k\al}|t|^{-2}R\cdot\norm{J''_\nu}_{L_{r\sim R}^2}\norm{J_\nu}_{L_{r\sim R}^2}\les 2^{-2k\al}|t|^{-2}R.
\end{align*}
The worst bound in the above integral is when two derivatives fall on the Bessel function. Thus eventually we get 
\begin{align*}
|K_E|\les (R^{-1/2}\lambda^{-5/2}+R^{-2})1_{2^{k\al}|t|\les BR}+2^{-2k\al}|t|^{-2}R\cdot 1_{2^{k\al}|t|\ges BR},
\end{align*}
where $B\gg 1$ is to be determined. From this bound, we get
\[\norm{K_E}_{L_t^1L^\infty_{r,r'}}\les  B2^{-k\al}(R^{1/2}\lambda^{-5/2}+R^{-1})+B^{-1}2^{-k\al}.\]
Taking $B=R^{-1/4}\lambda^{5/4}$, we get $\norm{K_E}_{L_t^1L^\infty_{r,r'}}\les R^{1/4}\lambda^{-5/4}2^{-k\al}$.  Thus we prove \eqref{eq:ER3infty}. By interpolation with \eqref{eq:ER3L2} we complete the proof. 
\end{proof}

Using Lemma \ref{lem:Tjkd3} with $\lambda=R^{1/2}$ we are able to prove Theorem \ref{thm} (2) in the dimension three and higher. Indeed, if $d\geq 3$ and $\frac{4d-2}{2d-3}<r<\frac{2d-2}{d-2}$, we have
\begin{align*}
&\norm{\chi_{\geq -k} (s)s^{\frac{d-1}{r}-\frac{d-2}{2}}T_k^\nu (h)}_{L_t^2L_s^r}\\
\les& \sum_{j\geq -k}2^{j(\frac{d-1}{r}-\frac{d-2}{2})}\norm{T_{j,k}^\nu (h)}_{L_t^2L_s^r}\\
\les& \sum_{m=1}^3\sum_{j\geq -k}2^{j(\frac{d-1}{r}-\frac{d-2}{2})}2^{-k/r}\norm{T_{j,k}^{\nu,m} (h)}_{L_t^2L_s^r}\\
\les&2^{-k/r}2^{-k\al/2}\sum_{j\geq -k}2^{j(\frac{d-1}{r}-\frac{d-2}{2})}\min([2^{-(j+k)/4}2^{k(\al-\be)/4}]^{1-2/r},1)\norm{h}_2\\
\les&2^{-k}2^{k(\frac{d}{2}-\frac{d}{r})}2^{-k\al/2}2^{k(\al-\be)(\frac{d-1}{r}-\frac{d-2}{2})}\norm{h}_2,
\end{align*}
and similarly for $r=\frac{2d-2}{d-2}$ we have
\begin{align}
\norm{\chi_{\geq -k} (s)s^{\frac{d-1}{r}-\frac{d-2}{2}}T_k^\nu (h)}_{L_t^2L_s^r}\les 2^{-k}\jb{k(\al-\beta)}2^{k(\frac{d}{2}-\frac{d}{r})}2^{-k\al/2}\norm{h}_2.
\end{align}

For $d=2$ and $6<r<\infty$, we use Lemma \ref{lem:Tjkd3} with $\lambda=R^{\frac{1}{3}+\epsilon}$. For $d=2$, $\nu\in \N$, the Bessel function have better decay estimates. For example, Lemma 2.3 in \cite{Guo} shows that if $\nu\in \N$, $\nu>r+\lambda$, and
$\lambda>r^{\frac{1}{3}+\e}$ for some $\e>0$, then for any $K\in \N$
\begin{align}
|J_\nu(r)|+|J_\nu'(r)|\leq C_{K,\e}r^{-K\e}.
\end{align}
With this we can get an improvement for $T_{j,k}^{\nu,2}$: for any $K\in \N$
\[
\norm{T_{j,k}^{\nu,2}  (h)}_{L_t^2L_s^r}\les 2^{-k\al/2}R^{-K}\norm{h}_2.
\] 
For $T_{j,k}^{\nu,3}$, the estimate on the error term is not good enough. We need an improvement by using Lemma \ref{lem:Bessel} part (2) for some large $K$ as in \cite{Guo}. We omit the details. 

\section{Proof of the main theorem}

In this section we prove Theorem \ref{thm:main}. To better illustrate our ideas, we only prove the theorem in the radial case.  In the general case, one can easily follow the techniques of integration in $SO(3)$ used in \cite{GLNW} to complete the proof.  The proof is based on applying Picard iteration in suitable spaces to
the following equivalent integral equations of \eqref{eq:GPm}
\EQ{\label{eq:GPint}
\begin{cases}
m=e^{-itH}\phi+\int_0^t e^{-i(t-s)H}[N_2+N_3+N_4+N_5](m,u)ds,\\
u_1=m_1-\frac{2u_1^2+u_2^2}{2-\De},\\
u_2=U^{-1}m_2,
\end{cases}
}
where $\phi=m(0)$ and
\EQ{
N_2(m,u)&=U(m_1^2)+\frac{2i}{2-\De}[-3m_1\De u_2-2\na m_1\cdot\na u_2],\\
N_3(m,u)&=U(2m_1R)+iN_3^1(u)+\frac{2i}{2-\De}[4u_1m_1u_2+ m_1^2u_2],\\
N_4(m,u)&=U(R^2 - |u|^4/4)+\frac{2i}{2-\De}[4u_1Ru_2+2u_2m_1R],\\
N_5(m,u)&=\frac{2i}{2-\De}[u_2R^2-u_2|u|^4/4].
}
We use the following resolution space
\begin{align*}
m\in X=&L_t^\infty L_x^2\cap L_t^{5/2}L_x^5\cap L_{t,x}^3 \cap D^{-1} (L_t^\infty L_x^2\cap L_{t,x}^3),\\
u_1\in Y=&L_t^{\infty}L_x^3\cap L_t^{5/2}L_x^5\cap L_t^{3}L_x^6\cap D^{-1}(L_t^{\infty}L_x^2\cap L_{t,x}^3),\\
u_2\in Z=&L_t^{5}L_x^{10}\cap D^{-1} (L_t^\infty L_x^2\cap L_{t,x}^3).
\end{align*}
Note that by interpolation and Sobolev embedding we have the embedding relation $X\subset Y\subset Z$. For simplicity of notations we write $S=Y\times Z$, and $u=(u_1,u_2)\in S$ with norm $\norm{u}_S=\norm{u_1}_Y+\norm{u_2}_Z$.
We use the dual space $N$ for the nonlinearity, with norm given by
\begin{align*}
\norm{F}_{N}=\norm{F}_{L^{3/2}_tH^1_{3/2}+ (L^1_tL^2_x\cap L^{3/2}_t\dot H^1_{3/2})}.
\end{align*}
Here $A+B$ denotes the standard sum space of two Banach spaces $A,B$.

\begin{lem}[Linear estimates]
We have the following estimate
\EQ{
\norm{e^{-itH}\phi}_{X}&\les \norm{\phi}_{H^1}\\
\normo{\int_0^t e^{-i(t-s)H}[F(s,\cdot)](x)ds}_{X}&\les \norm{F}_{N}
}
\end{lem}
\begin{proof}
The first inequality follows from Lemma \ref{lem:sphstri}. For the  inhomogeneous estimate we use the Christ-Kiselev lemma. Indeed, by this lemma, we immediately get: if $(q,r), (\wt q,\wt r)$ both satisfy the conditions in Corollary \ref{cor:sphstrqr} and $(q,\wt q)\neq (2,2)$, then
\begin{align}
\normo{\int_0^t e^{-i(t-s)H}P_k f(s)ds}_{L_t^qL_x^rL_\sigma^{2}(\R\times\R^3)}\les C_k(q,r)C_k(\wt q,\wt r)\norm{f}_{L_t^{\wt q'}L_x^{\wt r'}L_\sigma^{2}(\R\times\R^3)}
\end{align}
where $C_k(q,r)$ is given by \eqref{eq:Ckqr}. 
\end{proof}

\subsection{Nonlinear estimates} 
In this subsection we prove the crucial nonlinear estimates. 
Then main difficulty is the weak control on the low frequency component on $u_2$. All the nonlinear estimates are proved by paraproduct decompositions and H\"older inequalities.

\begin{lem}[Estimate for $u_1$]\label{lem:u1}
We have
\begin{align*}
\norm{(2-\Delta)^{-1}(u_2u_2)}_{Y}\les& \norm{u_2}_{Z}^2.
\end{align*}
\end{lem}
\begin{proof}
By Sobolev embedding and H\"older's inequalities we have
\begin{align*}
\norm{(2-\Delta)^{-1}(u_2u_2)}_{L_t^{\infty}L_x^3\cap L_t^{5/2}L_x^5}\les& \|u_2u_2\|_{L^\I_tL^3_x\cap L^{5/2}_tL^5_x} \\
\les&\norm{u_2}_{L^\I_tL^6_x \cap L^5_tL^{10}_x}\cdot \norm{u_2}_{L^\I_tL^6_x \cap L^5_tL^{10}_x}\\
\les&\norm{u_2}_{Z}\cdot \norm{u_2}_{Z},
\end{align*}
and
\begin{align*}
&\norm{D(2-\Delta)^{-1}(u_2u_2)}_{L_{t,x}^3}+\norm{(2-\Delta)^{-1}(u_2u_2)}_{L_t^{3}L_x^6}\\
&\les \|\na(u_2u_2)\|_{L^3_tL^2_x}
\les \|u_2\|_{L^\I_tL^6_x}\|\na u_2\|_{L^3_{t,x}}\les \norm{u_2}_{Z}^2,
\end{align*}
and
\begin{align*}
\norm{D(2-\Delta)^{-1}(u_2u_2)}_{L_t^{\infty}L_x^2}
\les& \norm{\na(u_2u_2)}_{L_t^{\infty}L_x^{3/2}}\\
\les& \|u_2\|_{L^\I_tL^6_x}\|\na u_2\|_{L_{t}^\infty L_x^2}\les\norm{u_2}_{Z}^2.
\end{align*}
Therefore, we complete the proof. 
\end{proof}

\begin{lem}[Estimate for $R$]\label{lem:R}
We have
\begin{align*}
\norm{R}_{L_t^3H^1\cap L_t^\infty L_x^{3/2}}\les& \norm{u_2}_{Z}^2+ \norm{u_1}_{Y}^2.
\end{align*}
\end{lem}
\begin{proof}
Since $R=\frac{-\De u_2^2}{2(2-\De)}-\frac{(2+\De)u_1^2}{2(2-\De)}$, then we have
\begin{align*}
\|R\|_{L^3_tH^1_x} \pt\lec \|u_2\na u_2\|_{L^3_tL^2_x}+\|u_1^2\|_{L^3_tH^1_x}
\pr\lec \|u\|_{L^\I_tL^6_x}\|\na u\|_{L^3_{t,x}}+\|u_1\|_{L^3_tL^6_x}\|u_1\|_{L^\I_tL^3_x}
\end{align*}
and
\EQN{
 \|R\|_{L^\I_tL^{3/2}_x} \pt\lec \|u_2\na u_2\|_{L^\I_tL^{3/2}_x}+\|u_1^2\|_{L^\I_tL^{3/2}_x} 
 \pr\lec \|u_2\|_{L^\I_tL^6_x}\|\na u_2\|_{L^\I_tL^2_x}+\|u_1\|_{L^\I_tL^3_x}^2.}
Therefore, we complete the proof. 
\end{proof}

\begin{lem}[Quadratic terms]
We have
\EQN{
\norm{N_2(m,u)}_{L^{3/2}_tH^1_{3/2}}\les \norm{m}_{X}^2+ \norm{m}_{X}\norm{u_2}_{Z}.
}
\end{lem}

\begin{proof}
Since $N_2(m,u)=U(m_1^2)-\frac{6i}{2-\De}[m_1\De u_2]-\frac{4i}{2-\De}[\na m_1\cdot\na u_2]$, we estimate the three terms separately. 
\EQN{
 \pt \|U(m_1)^2\|_{L^{3/2}_tH^1_{3/2}} \lec \|\na(m_1)^2\|_{L^{3/2}_{t,x}}
 \lec \|m_1\|_{L^3_{t,x}}\|\na m_1\|_{L^3_{t,x}}, 
 \pr \|(2-\De)^{-1}[\na m_1\cdot\na u_2]\|_{L^{3/2}_tH^1_{3/2}}
 \lec \|\na m_1\|_{L^3_{t,x}}\|\na u_2\|_{L^3_{t,x}}.}
For the remaining term, we use
\EQ{ \label{prod neg Sob}
 \|fg\|_{H^{-1}_{3/2}} \lec \|f\|_{H^1_3}\|g\|_{H^{-1}_3},}
which is the dual of 
\EQN{
 \|fg\|_{H^1_{3/2}} \lec \|f\|_{H^1_3}\|g\|_{H^1_3}.}
Then 
\EQN{
 \|(2-\De)^{-1}[m_1\De u_2]\|_{L^{3/2}_tH^1_{3/2}}
 \pt\lec \|m_1\De u_2\|_{L^{3/2}_tH^{-1}_{3/2}}
 \pr\lec \|m_1\|_{L^3_tH^1_3}\|\na u_2\|_{L^3_{t,x}}.}
Thus we complete the proof.
\end{proof}

\begin{lem}[Cubic terms]
\begin{align*}
\norm{N_3(m,u)}_{N}\les \norm{m}_{X}\norm{u}_{U}^2+\norm{m}_{X}^2 \norm{u}_{S}+\norm{u}_{S}^3.
\end{align*}
\end{lem}
\begin{proof}
Recall that $N_3(m,u)=U(2m_1R)+iN_3^1(u)+\frac{2i}{2-\De}[4u_1m_1u_2+ m_1^2u_2]$. The critical term is \EQ{
 N_3^c(u):=-u_2\frac{2+\De}{2-\De}u_1^2.}
The cubic terms $N_3-iN_3^c$ excepting the critical one are estimated in $L^{3/2}_tH^1_{3/2}$. We have
\EQN{
 \pt\|m_1R\|_{L^{3/2}_tH^1_{3/2}}
 \lec \|m_1\|_{L^3_tH^1_x}\|R\|_{L^3_{t,x}}
 +\|m_1\|_{L^3_tL^6_x}\|\na R\|_{L^3_tL^2_x},
 \pr\|(2-\De)^{-1}[u_1u_2m_1]\|_{L^{3/2}_tH^1_{3/2}}
 \lec \|u_1\|_{L^3_tL^6_x}\|u_2\|_{L^\I_tL^6_x}\|m_1\|_{L^3_{t,x}},}
and similarly for $(2-\De)^{-1}[m_1^2u_2]$. 
For $N_3^1-N_3^c$, we use 
\EQN{
 \pt\|fg\|_{H^{-1}_{3/2}} \lec \|f\|_{H^1}\|g\|_{H^{-1}_3},
 \pq\|fg\|_{H^{-1}_2} \lec \|f\|_{\dot H^1}\|g\|_{H^{-1}_3},}
which are proved by duality in the same way as \eqref{prod neg Sob}. 
Using $L^{6/5}_x\subset H^{-1}_{3/2}$ as well, 
\EQN{
 \|N_3^1-N_3^c\|_{L^{3/2}_tH^1_{3/2}}
 \pt\lec \|(\De u_2)(2-\De)^{-1}(\De u_2^2)\|_{L^{3/2}_tH^{-1}_{3/2}}
 \prQ+ \|(\na u_2)\cdot(2-\De)^{-1}\na(\De u_2^2)\|_{L^{3/2}_tL^{6/5}_x}
 \prQ+ \|(\De u_2)(2-\De)^{-1}(u_1^2)\|_{L^{3/2}_tH^{-1}_{3/2}}+ \|u_2|\na u_2|^2\|_{L^{3/2}_tL^{6/5}_x}
 \pr\lec \|\na u_2\|_{L^3_{t,x}}\Br{\|\De u_2^2\|_{L^3_tH^{-1}_2}+\|u_1^2\|_{L^3_tL^2_x}+\norm{u_2}_{L_t^\infty L_x^6}\norm{\na u_2}_{L^3_{t,x}}}
 \pr\lec \|\na u_2\|_{L^3_{t,x}}\Br{\|u_2\|_{L^\I_t\dot H^1_x}\|\na u_2\|_{L^3_{t,x}}+\|u_1\|_{L^\I_tL^6_x}\|u_1\|_{L^3_{t,x}}}. } 
The critical term is estimated in $L^1_tL^2_x\cap L^{3/2}_t\dot H^1_{3/2}$, by
\EQN{
 \pt\|N_3^c\|_{L^1_tL^2_x} \lec \|u_2\|_{L^5_tL^{10}_x}\|u_1^2\|_{L^{5/4}_tL^{5/2}_x}^2 \lec \|u_2\|_{L^5_tL^{10}_x}\|u_1\|_{L^{5/2}_tL^5_x}^2,}
and 
\EQN{
 \|\na N_3^c\|_{L^{3/2}_{t,x}} \pt\lec \|\na u_2\|_{L^3_{t,x}}\|u_1^2\|_{L^3_{t,x}}+ \|u_2\|_{L^\I_tL^6_x}\|u_1\na u_1\|_{L^{3/2}_tL^2_x}
 \pr\lec \Br{\|\na u_2\|_{L^3_{t,x}}\|u_1\|_{L^\I_tL^6_x}
 + \|u_2\|_{L^\I_tL^6_x}\|\na u_1\|_{L^3_{t,x}}}\|u_1\|_{L^3_tL^6_x}.
}
Thus we complete the proof.
\end{proof}

\begin{lem}[Quartic terms]
We have
\EQN{
\norm{N_4(m,u)}_{N}\les \norm{m}_{X} \norm{u}_{S}^3+\norm{u}_{S}^4.
}
\end{lem}
\begin{proof}

The quartic term $N_4$ is estimated in 
\EQN{
 \pt\|R^2\|_{L^{3/2}_tH^1_{3/2}} \lec \|R\|_{L^3_tH^1}^2, 
 \pr\|U|u|^4\|_{L^{3/2}_tH^1_{3/2}} \lec \|u^3\na u\|_{L^{3/2}_{t,x}}
 \lec \|\na u\|_{L^3_{t,x}}\|u^3\|_{L^3_{t,x}},}
where the last norm is bounded by Sobolev embedding and H\"older inequalities
\EQN{
 \|u^2\na u\|_{L^3_tL^{3/2}_x}
 \lec \|u\|_{L^\I_tL^6_x}^2\|\na u\|_{L^3_{t,x}}.}
The other terms are estimated in the same way as similar terms in $N_3$, such as 
\EQN{
 \|(2-\De)^{-1}[u_1u_2R]\|_{L^{3/2}_tH^1_{3/2}}
 \lec \|u_1\|_{L^3_tL^6_x}\|u_2\|_{L^\I_tL^6_x}\|R\|_{L^3_{t,x}}.}
We complete the proof.
\end{proof}

\begin{lem}[Quintic terms]
We have
\EQN{
\norm{N_5(m,u)}_{N}\les \norm{u}_{S}^5.
}
\end{lem}
\begin{proof}
For the quintic term $N_5$, the critical term is 
\EQN{
 \pq N_5^c(u):=-\frac{i}{8(2-\De)}[u_2|u|^4].}
The subcritical part is estimated as above 
\EQN{
 \|(2-\De)^{-1}[u_2R^2]\|_{L^{3/2}_tH^1_{3/2}}
 \lec \|u_2\|_{L^\I_tL^6_x}\|R\|_{L^3_{t,x}}\|R\|_{L^3_tL^6_x}.}
The critical term is estimated using $L^2_x+L^{6/5}_x\subset H^{-1}_x$ 
\EQN{
 \|N_5^c(u)\|_{L^1_tH^1_x}
 \pt\lec \|u_2|u|^4\|_{L^1_t(L^2_x+L^{6/5}_x)}
 \lec \|u_2^5\|_{L^1_tL^2_x}+\|u_2u_1^4\|_{L^1_tL^{6/5}_x}
 \pr\lec \|u_2\|_{L^5_xL^{10}_x}^5 + \|u\|_{L^\I_tL^6_x}^2\|u_1\|_{L^3_tL^6_x}^3.}
We complete the proof.
\end{proof}

\subsection{Proof of the theorem}

Now we prove Theorem \ref{thm:rad}. First we note that the transformation 
\begin{align*}
u\to m=m_1+im_2=T(u):=u_1 + \frac{2u_1^2+u_2^2}{2-\De}+iUu_2
\end{align*}
is a  homeomorphism between small balls with center $0$ in $(\E,d_\E)$ and $(H^1,\norm{\cdot}_{H^1})$ (The proof is similar to \cite{GNT3}). Thus if $u_0\in \E$ with $E(u_0)\ll 1$ then we have $m(0)\in H^1$ with a small norm.  

Fix $m_0=m(0)$, we define an operator $\Phi_{m_0}(m,u_1,u_2)$ by
the right-hand side of \eqref{eq:GPint}. Our
resolution space is
\[W_\eta=\{(m,u_1,u_2): \norm{(m,u_1,u_2)}_W=\norm{m}_{X}+\frac{1}{100}\norm{u_1}_{Y}+\frac{1}{100}\norm{u_2}_{Z}\leq \eta\}\]
endowed with the norm metric $\norm{\cdot}_W$.

By the linear estimates and the nonlinear estimates proved in the previous subsection, we can easily show that $\Phi_{m_0}:W_\eta\to W_\eta$ is a
contraction mapping by the condition $\norm{m_0}_{H^1}\ll 1$ and choosing suitable $\eta\ll 1$. So we get a unique solution $(m,u)$ in $W_\eta$. Moreover, since our estimates are time global, by the standard techniques we get $m$ scatters in $H^1$. Namely, $\exists \phi_\pm\in H^1$ such that
\[\lim_{t\to \pm \infty}\norm{m-e^{-itH}\phi_\pm}_{H^1}=0.\]
By the second and third equation in \eqref{eq:GPint} we get
\begin{align}\label{eq:scatu}
\lim_{t\to \pm \infty}\norm{u_1-\re(e^{-itH}\phi_\pm)}_{\dot H^1\cap \dot H^{1/2}}+\norm{u_2-U^{-1}\im(e^{-itH}\phi_\pm)}_{\dot H^1}=0.
\end{align}
Using the transformation $m=T(u)$ we get
\begin{align*}
2m_1=2u_1+|u|^2+\frac{\Delta |u|^2}{2-\Delta}+\frac{2u_1^2}{2-\Delta}.
\end{align*}
By the similar arguments as for Lemma \ref{lem:u1}, we can prove
\begin{align}
\normo{\frac{\Delta |u|^2}{2-\Delta}}_{L_t^\infty L_x^2}+\normo{\frac{2u_1^2}{2-\Delta}}_{L_t^\infty L_x^2}\les \norm{u_1}_{Y}^2+\norm{u_2}_{Z}^2. 
\end{align}
With this and \eqref{eq:scatu} we get
\[
\lim_{t\to \pm \infty}\normo{\frac{\Delta |u|^2}{2-\Delta}}_{L_x^2}+\normo{\frac{2u_1^2}{2-\Delta}}_{L_x^2}=0,\]
and then we get 
\[\lim_{t\to \pm \infty}\norm{2u_1+|u|^2-2\re(e^{-itH}\phi_\pm)}_{L^2}=0.\]
Therefore we complete the proof of Theorem \ref{thm:rad}.


\begin{thebibliography}{99}
\bibitem{BSaut} F. B\'ethuel, J.-C. Saut, Travelling waves for the Gross-Pitaevskii equation I, Ann. Inst. H. Poincar\'e Phys. Th\'eor. 70 (1999), no. 2, 147--238.

\bibitem{Chiron} D. Chiron, Travelling waves for the Gross-Pitaevskii equation in dimension larger than two.
Nonlinear Anal. 58 (2004), no. 1-2, 175--204.

\bibitem{BGS}
F. Bethuel, P. Gravejat and J. C. Saut, Travelling waves for the Gross-Pitaevskii equation II,
Comm. Math. Phys. 285 (2009), no. 2, 567--651.

\bibitem{FS} A. L. Fetter and A. A. Svidzinsky, Vortices in a trapped dilute Bose-Einstein condensate, J. Phys. Condens. Matter, 13 (2001), R135--R194.

\bibitem{Gerard} P. G\'erard, The Cauchy problem for the Gross-Pitaevskii equation. Ann. Inst. H. Poincar\'e Anal. Non Lin\'eaire 23 (2006), no. 5, 765--779.

\bibitem{Guo} Z. Guo, Sharp spherically averaged Strichartz estimates for the Schr\"odinger equation, Nonlinearity 29 (2016), 1668--1686.

\bibitem{GN}
Z. Guo, K. Nakanishi, Small energy scattering for the Zakharov system with radial symmetry, Int. Math. Res. Not. 9 (2014), 2327--2342.

\bibitem{GNW}
Z. Guo, K. Nakanishi and S. Wang, Small energy scattering for the Klein-Gordon-Zakharov system with radial symmetry, Math. Res. Let., 21 (2014), no.4, 733--755.
\bibitem{GLNW} Z. Guo, S. Lee, K. Nakanishi and C. Wang, Generalized Strichartz Estimates and Scattering
for 3D Zakharov System, Comm. Math. Phy. 331 (2014), no. 1, 239--259.

\bibitem{GPW}
Z. Guo, L. Peng and B. Wang, Decay estimates for a class of wave equations, Journal of Functional Analysis, 254 (2008), 1642--1660.

\bibitem{GNT1}
S. Gustafson, K. Nakanishi and T.-P. Tsai, Scattering theory for the Gross-Pitaevskii equation, Math. Res. Lett., 13 (2006), no. 2, 273--285.

\bibitem{GNT2}
S. Gustafson, K. Nakanishi and T.-P. Tsai, Global dispersive solutions for the Gross-Pitaevskii equation in two and three dimensions, Ann. Henri Poincar\'e 8 (2007), 1303--1331.

\bibitem{GNT3}
S. Gustafson, K. Nakanishi and T.-P. Tsai, Scattering theory for the Gross-Pitaevskii equation in three dimensions, Comm. Contem. Math., 4 (2009), 657--707.

\bibitem{Stein1}{E. Stein, G. Weiss, Introduction to Fourier Analysis on Euclidean Spaces, Princeton University Press, 1971.}
\bibitem{Stein2} {E. M. Stein,  Harmonic analysis: real-variable methods, orthogonality, and oscillatory integrals, Princeton University Press, 1993.}
\bibitem{Watson}{G. Watson, A treatise on the theory of Bessel functions, Reprint of the second (1944) edition.
Cambridge University Press, Cambridge, 1995.}
\end{thebibliography}
\end{document}